\documentclass[12pt,a4paper]{amsart}
\usepackage{amsfonts}
\usepackage{amsthm}
\usepackage{amsmath}
\usepackage{amscd}
\usepackage[latin2]{inputenc}
\usepackage{graphicx}
\numberwithin{equation}{section}
\usepackage[margin=2.6cm]{geometry}
\usepackage{dsfont}
\usepackage{enumitem}
\usepackage{color}
\usepackage[toc]{appendix}
\usepackage{bm}

\makeatletter
\newcommand{\xRightarrow}[2][]{\ext@arrow 0359\Rightarrowfill@{#1}{#2}}
\makeatother

\makeatletter
\newcommand*{\rom}[1]{\expandafter\@slowromancap\romannumeral #1@}
\makeatother
\newcommand{\kc}{\mathcal{C}}

\theoremstyle{plain}
\newtheorem{theorem}{Theorem}[section]
\newtheorem{Lemma}[theorem]{Lemma}
\newtheorem{Cor}[theorem]{Corollary}
\newtheorem{Prop}[theorem]{Proposition}

\theoremstyle{definition}

\newtheorem{Rem}[theorem]{Remark}
\newtheorem{?}[theorem]{Problem}
\newtheorem{Ex}[theorem]{Example}

\usepackage{hyperref}
\begin{document}
	\title[Convergence and mass distributions of multivariate Archimedean copulas]{On convergence and mass 
	distributions of multivariate Archimedean copulas and their interplay with the Williamson transform}
	\author[]{Thimo M. Kasper, Nicolas Dietrich and Wolfgang Trutschnig}
	\address{Thimo M. Kasper\\ University of Salzburg \\ Department for Artificial Intelligence and Human Interfaces \\ Hellbrunnerstrasse 34, 5020 Salzburg}
	\email{thimo.kasper@plus.ac.at}
		\address{Nicolas Dietrich\\ University of Salzburg \\ Department for Artificial Intelligence and Human Interfaces \\ Hellbrunnerstrasse 34, 5020 Salzburg}
	\email{nicolaspascal.dietrich@plus.ac.at}
	\address{Wolfgang Trutschnig\\ University of Salzburg \\ Department for Artificial Intelligence and Human Interfaces \\ Hellbrunnerstrasse 34, 5020 Salzburg}
	\email{wolfgang.trutschnig@plus.ac.at}
	
	\begin{abstract} 
		Motivated by a recently established result saying that within the class of bivariate Archimedean copulas 
		standard pointwise convergence implies weak convergence of almost all conditional distributions 
		this contribution studies the class $\mathcal{C}_{ar}^d$ of all $d$-dimensional Archimedean copulas  
		with $d \geq 3$ and proves the afore-mentioned implication with respect to conditioning on the 
		first $d-1$ coordinates.
		Several proper\-ties equivalent to pointwise convergence in $\mathcal{C}_{ar}^d$ 
		are established and - as by-product of working with conditional distributions (Markov kernels) - alternative simple 
		proofs for the well-known formulas for the level set masses $\mu_C(L_t)$ and the Kendall distribution function 
		$F_K^d$ as well as a novel geometrical interpretation of the latter are provided.   
		Viewing normalized generators $\psi$ of $d$-dimensional Archimedean copulas from the perspective 
		of their so-called Williamson measures $\gamma$ on $(0,\infty)$ is then shown to allow not only to derive 
		surprisingly simple expressions for $\mu_C(L_t)$ and $F_K^d$ in terms of $\gamma$ and to characterize  
		pointwise convergence in $\mathcal{C}_{ar}^d$ by weak convergence of the Williamson measures 
		but also to prove that regularity/singularity properties of $\gamma$ directly carry over to the 
		corresponding copula $C_\gamma \in \mathcal{C}_{ar}^d$. These results are finally used to 
		prove the fact that the family of all absolutely continuous and the family of all 
		singular $d$-dimensional copulas is dense in $\mathcal{C}_{ar}^d$ and to underline that 
		despite of their simple algebraic structure Archimedean copulas may exhibit surprisingly singular behavior 
		in the sense of irregularity of their conditional distribution functions. 
	\end{abstract}
	
	\maketitle
	\tableofcontents
	\vspace*{-0.6cm}

\section{Introduction}
\label{sec:1:intro}

Archimedean copulas are a well-known family of dependence models whose popularity is mainly due to their simple algebraic 
structure: given a so-called (Archimedean, sufficiently monotone) 
generator $\psi: [0,\infty)\to[0,1]=:\mathbb{I}$ and letting $\varphi$ denote its
pseudo-inverse the Archimedean copula $C_\psi$ is defined by
$$
C_\psi(x_1,\ldots,x_d) = \psi(\varphi(x_1) + \cdots + \varphi(x_d)).
$$
As a consequence, analytic, dependence, and convergence properties of Archimedean co\-pulas can be characterized in 
terms of the corresponding generators (see, e.g., \cite{convArch,p21,wcc, multiArchNeslehova,Nelsen}). 
In particular, it was recently shown in \cite{wcc} that within the class of \emph{bivariate} Archimedean copulas 
pointwise convergence is equivalent to uniform convergence of the corresponding generators and, more importantly, 
even implies weak convergence of almost all conditional distributions (a.k.a weak conditional convergence, 
a concept generally much stronger than pointwise convergence).  
This result is surprising insofar that given samples $(X_1,Y_1), (X_2, Y_2), \ldots$ from some bivariate copula $C$ the corresponding sequence of empirical copulas $(\hat{E}_n)_{n\in\mathbb{N}}$ 
does not necessarily converge weakly conditional to $C$.\\
The focus of the current paper is twofold: \emph{on the one hand} we study convergence
in the class $\mathcal{C}_{ar}^d$ of all $d$-dimensional Archimedean copulas, $d \geq 3$,
and show that most results from the bivariate setting as established in \cite{wcc} also 
hold in $\mathcal{C}_{ar}^d$, including the surprising fact that pointwise convergence implies weak conditional 
convergence (whereby we consider conditioning on the first $d-1$ coordinates). 
As a nice by-product of working with Markov kernels (conditional distributions) we obtain simple, alternative proofs 
of the well-known formulas for the Kendall distribution function  $F_K^d$ and the level set masses of Archimedean copulas. 
To the best of the authors' knowledge, working with so-called $\ell_1$-norm symmetric distributions 
(as studied in \cite{multiArchNeslehova}) nowadays seems to be the standard approach for deriving these formulas in the 
multivariate setting - we show that working with Markov kernels constitutes an interesting alternative 
and may provide additional insight.  
Additionally, motivated by \cite[Chapter 4.3]{Nelsen} and \cite[Section 3]{p21} we offer a see\-mingly novel 
geometric interpretation of the level set masses in terms of $\psi$. \\
And, \emph{on the other hand}, we revisit the close interrelation between Archimedean copulas and  
probability measures $\gamma$ on $(0,\infty)$ via the so-called Williamson transform as studied in 
\cite{multiArchNeslehova} (also see \cite{schilling2012}), characterize properties of the generator 
$\psi$ in terms of normalized $\gamma$ (to which we will refer to as Williamson measure) and 
then prove the fact that pointwise convergence in $\mathcal{C}_{ar}^d$ is equivalent to weak 
convergence of the correspon\-ding probability measures on $(0,\infty)$. 
Moreover, we derive surprisingly simple expressions for the 
level set masses and the Kendall distribution functions in terms of $\gamma$ and then show that 
singularity/regularity properties of $\gamma$ directly carry over to the corresponding Archimedean copula 
$C_\gamma \in \mathcal{C}_{ar}^d$. This very property is then used, firstly, to show that the family of 
absolutely continuous as well as two disjoint subclasses of the family of all singular Archimedean copulas are 
dense in $\mathcal{C}_{ar}^d$ and, secondly, to illustrate the fact that despite their simple and handy 
algebraic structure Archimedean copulas may exhibit surpri\-singly irregular behavior by constructing 
elements of $\mathcal{C}_{ar}^d$ which have full support 
although being singular (see \cite{p21} for the already established bivariate results).  \\

The rest of this contribution is organized as follows: Section \ref{sec:2:preliminaries} contains notation and preliminaries that are used throughout the text. Section \ref{sec:3:mass:distri} starts with deriving 
an explicit expression for the Markov kernel of $d$-dimensional Archimedean copulas, then restates the well-known 
formulas for the masses of level sets as well as the Kendall distribution function, and provides a 
geometric interpretation for the latter in terms of the generator $\psi$. 
In Section \ref{sec:convergence} we derive various characterizations of pointwise convergence in 
$\mathcal{C}_{ar}^d$ in several steps and prove the main result saying that pointwise convergence 
implies weak conditional convergence (with respect to the first $d-1$ coordinates). 
Turning to the Williamson transform, in Section \ref{sec5:williamson} we first establish some 
complementing results on the interrelation of the generator $\psi$ and the Williamson measure $\gamma$, then 
characterize pointwise convergence in $\mathcal{C}_{ar}^d$ in terms of weak convergence of the Williamson measures, and finally show how regularity/singularity properties of $\gamma$ carry over to $C_\gamma$. These properties are
then used to prove the afore-mentioned denseness results and to construct singular Archimedean copulas 
with full support. Several examples and graphics illustrate the obtained results.

\section{Notation and preliminaries}
\label{sec:2:preliminaries}

In the sequel we will let $\mathcal{C}^d$ denote the family of all $d$-dimensional copulas for some fixed integer $d\geq 3$ and write vectors in bold symbols. For each copula $C\in\mathcal{C}^d$ the corresponding $d$-stochastic measure will be
denoted by $\mu_C$, i.e., $\mu_C([\mathbf{0},\mathbf{x}]) = C(\mathbf{x})$ 
for all $\mathbf{x} \in \mathbb{I}^d$, whereby $[\mathbf{0},\mathbf{x}] := [0,x_1] \times [0,x_2] \times\ldots\times [0,x_d]$ and $\mathbb{I} := [0,1]$. To notation as simple as possible we will frequently write 
$\mathbf{x}_{1:m}=(x_1,\ldots,x_m)$ for $\mathbf{x} \in \mathbb{I}^d$ and $m \leq d$.
Considering $1\leq i < j \leq d$, the $i$-$j$-marginal of $C$ will be denoted by $C^{ij}$, i.e., we have 
$C^{ij}(x_i, x_j) = C(1,\ldots, 1,x_i,1,\ldots,1,x_j,1,\ldots,1)$. In order to keep notation as simple as possible
for every $m<d$ the marginal copula of the first $m$ coordinates will be denoted by $C^{1:m}$, i.e., 
$C^{1:m}(x_1,x_2,\ldots, x_m) = C(x_1,x_2,\ldots, x_m, 1,\ldots, 1)$. 
Considering the uniform metric $d_\infty$ on $\mathcal{C}^d$ it is well-known that $(\mathcal{C}^d, d_\infty)$ is a 
compact metric space and that in $\mathcal{C}^d$ pointwise and uniform convergence are equivalent. 
For more background on copulas and $d$-stochastic measures we refer to \cite{Principles, Nelsen}. 

For every metric space $(S,d)$ the Borel $\sigma$-field on $S$ will be denoted by $\mathcal{B}(S)$ and 
$\mathcal{P}(S)$ denotes the family of all probability measures on $\mathcal{B}(S)$. 
The Lebesgue measure on $\mathcal{B}(\mathbb{I}^d)$ will be denoted by $\lambda$ or (whenever particular emphasis to 
the dimension $d$ is required) by $\lambda_{d}$. Furthermore $\delta_x$ denotes the Dirac measure in $x \in S$.

In what follows Markov kernels will play a prominent role. For $m<d$ an $m$-\emph{Markov kernel} from $\mathbb{R}^m$ to $\mathbb{R}^{d-m}$ is a mapping $K: \mathbb{R}^m\times\mathcal{B}(\mathbb{R}^{d-m}) \rightarrow \mathbb{I}$ fulfilling
that for every fixed $E\in\mathcal{B}(\mathbb{R}^{d-m})$ the mapping $\mathbf{x}\mapsto K(\mathbf{x},E)$ is $\mathcal{B}(\mathbb{R}^{m})$-$\mathcal{B}(\mathbb{R}^{d-m})$-measurable and for every fixed $\mathbf{x}\in\mathbb{R}^m$ the mapping $E\mapsto K(\mathbf{x},E)$ is a probability measure on $\mathcal{B}(\mathbb{R}^{d-m})$. 
Given a $(d-m)$-dimensional random vector $\mathbf{Y}$ and an $m$-dimensional random vector $\mathbf{X}$ on a 
probability space $(\Omega, \mathcal{A}, \mathbb{P})$
we say that a Markov kernel $K(\cdot,\cdot)$ is a regular conditional distribution of
$\mathbf{Y}$ given $\mathbf{X}$ if for every fixed $E \in \mathcal{B}(\mathbb{R}^{d-m})$ the identity 
$$
K(\mathbf{X}(\omega), E) = \mathbb{E}(\mathbf{1}_E \circ \mathbf{Y} | \mathbf{X})(\omega)
$$
holds for $\mathbb{P}$-almost every $\omega \in \Omega$.  
It is well-known that for each pair of random vectors $(\mathbf{X}, \mathbf{Y})$ as 
above, a regular conditional distribution $K(\cdot, \cdot)$ of $\mathbf{Y}$ given $\mathbf{X}$ exists and is 
unique for $\mathbb{P}^{\mathbf{X}}$-almost all $\mathbf{x} \in \mathbb{R}^m$, 
whereby as usual $\mathbb{P}^{\mathbf{X}}$ denotes the push-forward of $\mathbb{P}$ via $\mathbf{X}$.
In case $(\mathbf{X}, \mathbf{Y})$ has $C\in\mathcal{C}^d$ as distribution function (restricted to $\mathbb{I}^d$) we let $K_{C}:\mathbb{I}^m \times \mathcal{B}(\mathbb{I}^{d-m}) \to \mathbb{I}$ denote (a version of) the regular conditional distribution of $\mathbf{Y}$ given $\mathbf{X}$ and simply refer to it as $m$-\emph{Markov kernel} of $C$.
Defining the $\mathbf{x}$-section $G_\mathbf{x}$ of a set $G \in\mathcal{B}(\mathbb{I}^d)$ w.r.t. the first $m$ coordinates by $G_{\mathbf{x}}:=\{\mathbf{y}
\in \mathbb{I}^{d-m}: (\mathbf{x},\mathbf{y}) \in G\}\in\mathcal{B}(\mathbb{I}^{d-m})$ the well-known 
\emph{disintegration theorem} implies
\begin{align*}
	\mu_C(G) = \int_{\mathbb{I}^m} K_{C}(\mathbf{x},G_{\mathbf{x}})
	\, \mathrm{d}\mu_{C^{1:m}}(\mathbf{x}).
\end{align*}
It is well-known that the disintegration theorem also holds for general finite measures in which case the 
conditional measures $K(\cdot,\cdot)$ are not necessarily probability measures but general finite measures 
(sub- or super Markov kernels). For more background on conditional expectation and disintegration we 
refer to the \cite{Kallenberg} and \cite{Klenke}. 

An Archimedean \emph{generator} $\psi$ is a continuous, non-increasing function $\psi: [0,\infty) \to [0,1]$ fulfilling
$\psi(0) = 1$, $\lim_{z\to\infty} \psi(z) = 0=:\psi(\infty)$ and being strictly decreasing on $[0,\inf\{z\in[0,\infty]: \psi(z) = 0\}]$ (with the convention $\inf \emptyset :=\infty$). 
For every Archimedean generator $\psi$ we will let $\varphi: [0,1] \rightarrow [0,\infty]$ 
denote its pseudo-inverse defined by
$\varphi(y) = \inf\{z\in[0,\infty]: \psi(z) = y\}$ for every $y \in [0,1]$. Obviously $\varphi$ is strictly decreasing 
on $[0,1]$ and fulfills $\varphi(1)=0$, moreover it is straightforward to verify that $\varphi$ is right-continuous 
at $0$ (for a short discussion of this property see Section 4 in \cite{wcc}).
If $\varphi(0)=\infty$ we refer to $\psi$ (and $\varphi$) as \emph{strict} and as \emph{non-strict} otherwise.  
A copula $C\in\mathcal{C}^d$ is called \emph{Archimedean} (in which case we write 
$C\in\mathcal{C}_\text{ar}^d$) if there exists some Archimedean generator $\psi$ with
\begin{align*}
	C_\psi(\mathbf{x}) = \psi(\varphi(x_1) + \cdots + \varphi(x_d))
\end{align*}
for every $\mathbf{x}\in\mathbb{I}^{d}$. 
Following \cite{multiArchNeslehova}, $C_\psi(\mathbf{x}) = \psi(\varphi(x_1) + \cdots + \varphi(x_d))$ 
is a $d$-dimensional copula if, and only if, $\psi$ is a $d$-monotone Archimedean 
generator on $[0,\infty)$, i.e., if, and only if, 
$\psi$ is an Archimedean generator fulfilling that $(-1)^{d-2}\psi^{(d-2)}$ exists on $(0,\infty)$, is non-negative, 
non-increasing and convex on $(0,\infty)$, whereby, as usual, $g^{(m)}$ denotes the $m$-th derivative of a function $g$.
Moreover (again see \cite{multiArchNeslehova}) it is straightforward to verify that in the latter 
case $(-1)^{m}\psi^{(m)}$ exists on $(0,\infty)$, 
is non-negative, non-increasing and convex on $(0,\infty)$ for every $m \in \{0,\ldots,d-2\}$.
In the sequel we will sometimes simply write $C$ instead of $C_\psi$ when no confusion may arise. 
Furthermore in the sequel we will simply refer to Archimedean generators as `generators'. 

Letting $D^-g$ and $D^+g$ denote the left- and right- hand derivative of a function $g$, respectively, 
convexity of $(-1)^{d-2}\psi^{(d-2)}$ implies that both, $D^-\psi^{(d-2)}(z)$ and $D^+\psi^{(d-2)}(z)$ exist
for every $z\in(0,\infty)$ and that the two derivatives coincide outside an at most countable set (see \cite{KK,pollard}) 
- in fact, for every
continuity point $z$ of $D^-\psi^{(d-2)}$ we have $D^-\psi^{(d-2)}(z)=D^+\psi^{(d-2)}(z)$.
Moreover, every $d$-monotone generator $\psi$ fulfills $\lim_{z \rightarrow \infty} \psi^{(m)}(z) = 0$ for every
$m\in \{0,\ldots,d-2\}$ as well as $\lim_{z \rightarrow \infty} D^{-}\psi^{(d-2)}(z) = 0$. 
Indeed, $\lim_{z \rightarrow \infty} \psi'(z) = 0$ directly follows from monotonicity and convexity of $\psi$ since
$d$-monotonicity implies that $-\psi'$ is decreasing and convex too; proceeding iteratively yields the assertion. 
Notice that Lemma \ref{lem:Left_hand_G} yields an even simpler direct proof of this assertion.
In the sequel we will also use the fact that (again by convexity, see \cite{KK,pollard})
we can reconstruct the generator $\psi$ from its derivatives in the sense that ($m \in \{1,\ldots,d-2\}$)
\begin{align}\label{eq:up}
\psi^{(m-1)}(z) &= \int_{[z,\infty)} -\psi^{(m)}(s) \ \mathrm{d}\lambda(s) , \quad 
\psi^{(d-2)}(z) = \int_{[z,\infty)} -D^-\psi^{(d-2)}(s) \ \mathrm{d}\lambda(s).
\end{align}
In order to have a one-to-one correspondence between copulas and their generator we follow \cite{wcc} and 
from now on implicitly assume that all generators are normalized in the sense that $\varphi(\frac{1}{2}) = 1$, 
or equivalently, $\psi(1) = \frac{1}{2}$ holds.

According to \cite{multiArchNeslehova}, an Archimedean copula $C\in\mathcal{C}_\text{ar}^d$ is absolutely continuous if, 
and only if, $\psi^{(d-1)}$ exists and is absolutely continuous on $(0,\infty)$. In this case a version of the density 
$c$ of $C$ is given by 
\begin{align}\label{eq:arch:density}
	c(\mathbf{x}) = \mathbf{1}_{(0,1)^d}(\mathbf{x}) \prod_{i=1}^{d}\varphi'(x_i) \cdot D^-\psi^{(d-1)}\big( \varphi(x_1)+\cdots +\varphi(x_d) \big).
\end{align}
In the sequel we will use the handy consequence that lower dimensional marginals of $d$-dimensional Archimedean 
copulas are absolutely continuous (\cite[Proposition 4.1]{multiArchNeslehova}).
\newpage

\section{Markov kernel, mass distribution and Kendall distribution function of multivariate Archimedean copulas}
\label{sec:3:mass:distri}
In the proceedings contribution \cite{multiArchMassKernel} the authors derive an explicit expression for 
(a version of) the $(d-1)$-Markov kernel of $d$-variate Archimedean copulas which, in turn, allows to derive the 
well-known formulas for level-set mass and the Kendall distribution function of $d$-variate Archimedean copulas
in an alternative way. Considering that Markov kernels are key for the rest of this paper we introduce a 
(slightly modified) 
version of the Markov kernel considered in \cite{multiArchMassKernel} and prove its properties in detail. 
Furthermore, we restate the already known formulas for the Kendall distribution function and level-set mass, 
add a new geometric interpretation, and, for the sake of completeness, include the corresponding purely 
Markov kernel-based proofs in the Appendix.

For every $t\in(0,1]$ define the $t$-level set of $C\in\mathcal{C}_\text{ar}^d$ by
\begin{equation}\label{eq:Lt}
	L_t := \left\lbrace (\mathbf{x},y)\in\mathbb{I}^{d-1}\times\mathbb{I}: C(\mathbf{x},y) = t \right\rbrace = \left\lbrace (\mathbf{x},y)\in \mathbb{I}^{d-1}\times \mathbb{I}: \sum_{i=1}^{d-1}\varphi(x_i) + \varphi(y) = \varphi(t) \right\rbrace
\end{equation}
and for $t=0$ set 
\begin{equation}\label{eq:L0}
	L_0 := \left\lbrace (\mathbf{x},y)\in\mathbb{I}^{d-1}\times\mathbb{I}: C(\mathbf{x},y) = 0 \right\rbrace =	\left\lbrace (\mathbf{x}, y)\in \mathbb{I}^{d-1}\times \mathbb{I}: \sum_{i=1}^{d-1} \varphi(x_i) + \varphi(y) \geq \varphi(0) \right\rbrace.
\end{equation}
Subsequently we will work with the level sets of the $(d-1)$-dimensional marginal of $C$ defined analogously 
and denote them by $L_t^{1:d-1}$ and $L_0^{1:d-1}$, respectively. 
As in the bivariate setting (see \cite{wcc}) we can define functions $f^t$ whose graph coincides with $L^t$ for 
$t \in (0,1]$ and with the boundary of $L^0$ for $t=0$: 
In fact, for $t=0$ defining the function $f^0: \mathbb{I}^{d-1}\to\mathbb{I}$ by
\begin{align*}
	f^0(\mathbf{x}) = \begin{cases}
		1 & \text{ if } \mathbf{x}\in L_0^{1:d-1}\\
		\psi\left( \varphi(0)-\sum_{i=1}^{d-1}\varphi(x_i) \right) & \text{ if } \mathbf{x}\not\in L_0^{1:d-1}
	\end{cases}
\end{align*}
with the conventions $\psi(\infty)=0$ as well as $\psi(u) = 1$ for all $u<0$ and for $t\in (0,1]$, defining 
the upper $t$-cut of 
$C^{1:d-1}$ by $[C^{1:d-1}]_t = \lbrace \mathbf{x}\in\mathbb{I}^{d-1}: C^{1:d-1}(\mathbf{x})\geq t\}$
and considering the function 
$f^t: [C^{1:d-1}]_t\to\mathbb{I}$ given by 
\begin{align*}
	f^t(\mathbf{x}) := \psi\left(\varphi(t) - \sum_{i=1}^{d-1}\varphi(x_i)\right).
\end{align*}
yields the above-mentioned property. It is straightforward to verify that 
for $\mathbf{x}\not\in$ $L_0^{1:d-1}$ and $y < f^0(\mathbf{x})$ we have $(\mathbf{x},y)\in L_0$ and that for strict Archimedean copulas $\mathbf{x}\in L_0^{1:d-1}$ if, and only if, 
$M(\mathbf{x}) = 0$ where $M$ denotes the $d$-dimensional minimum copula. 

\begin{theorem}\label{th:multi:kernel}
	Suppose that $C \in\mathcal{C}_\text{ar}^d$ has generator $\psi$. Then setting 
\begin{align}\label{eq:mk1}
	K_C(\mathbf{x}, [0,y]) := \begin{cases}
		1, & M(\mathbf{x}) = 1 \emph{ or } \mathbf{x}\in L_0^{1:d-1} \\
		0, & M(\mathbf{x})<1, \mathbf{x}\not\in L_0^{1:d-1}, y< f^0( \mathbf{x}) \\
		\frac{D^-\psi^{(d-2)}\left(\sum\limits_{i=1}^{d-1}\varphi(x_i) + \varphi(y)\right)}{D^-\psi^{(d-2)}\left(\sum\limits_{i=1}^{d-1}\varphi(x_i)\right)}, & M(\mathbf{x})<1, \mathbf{x}\not\in L_0^{1:d-1}, y\geq f^0(\mathbf{x})
	\end{cases}
\end{align}
	yields (a version of) the $(d-1)$-Markov kernel of $C$.
\end{theorem}
\begin{proof}
	First notice that absolute continuity of $C^{1:d-1}$ implies 
	\begin{align*}
		0=\mu_{C^{1:d-1}}(L_0^{1:d-1}) &= \mu_C(L_0^{^:d-1}\times\mathbb{I}).
	\end{align*}
	so $\mu_{C^{1:d-1}}(L_0^{1:d-1} \cup M^{-1}(\{1\}))=0$, so the first condition 
	on the right hand side of eq. \eqref{eq:mk1} only holds for a set of $\mu_{C^{1:d-1}}$-measure $0$. 
	
	We start by showing that $K_C$ defined according to eq. \eqref{eq:mk1} is indeed a $(d-1)$-Markov kernel
	and then prove that it is a Markov kernel of $C$. Fix $y\in\mathbb{I}$. Since measurability of $f^0$ and 
	$D^-\psi^{(d-2)}$ are a direct consequence of the properties of $\psi$, using continuity of $\varphi$ yields
	measurability of the mapping $\mathbf{x}\mapsto K_C(\mathbf{x},[0,y])$. Building upon that, applying 
	using a standard Dynkin System argument (see \cite[Theorem 2]{p21}) yields measurability of 
	$\mathbf{x}\mapsto K_C(\mathbf{x},E)$ for every Borel set $E\in\mathcal{B}(\mathbb{I})$. \newline
	As for $\mathbf{x} = 1$ or fixed $\mathbf{x}\in L_0^{1:d-1}$ the map $y\mapsto K_C(\mathbf{x},[0,y])$ is obviously 
 a univariate distribution function consider $\mathbf{x}\not\in L_0^{1:d-1}$. Then monotonicity and left-continuity 
 of $(-1)^{d-2} D^-\psi^{(d-2)}$ implies that $y\mapsto K_C(\mathbf{x}, [0,y])$ is increasing and right-continuous. 
 Moreover we obviously have $K_C(\mathbf{x},[0,1])=1$, so $y\mapsto K_C(\mathbf{x}, [0,y])$ is a univariate 
 distribution function and it remains to show that it is a $(d-1)$-Markov kernel of $C$, i.e., that 
 \begin{align}\label{disint.check}
		C(\mathbf{x},y) &=\int_{[\mathbf{0},\mathbf{x}]} K_C(\mathbf{s},[0,y]) \mathrm{d}\mu_{C^{1:d-1}}(\mathbf{s})
	\end{align}
 holds for all $\mathbf{x} \in \mathbb{I}^{d-1}$ and every $y \in \mathbb{I}$. \newline
	The case $y=1$ is trivial and for $y=0$ we have that $K_C(\mathbf{x},\{0\})=0$ for 
	$\mu_{C^{1:d-1}}$-almost every $\mathbf{x}$, implying that  
	eq. (\ref{disint.check}) holds. 
	 For $y\in(0,1)$ considering that $L_0^{1:d-1}$ is a $\mu_{C^{1:d-1}}$-null set and 
	 using absolute continuity of $\mu_{C^{1:d-1}}$ yields ($\Upsilon:=\mathbb{I}^{d-1}\setminus L_0^{1:d-1}$)\\
	 \begin{align*}
		I &:= \int_{[\mathbf{0},\mathbf{x}]\cap \Upsilon} K_C(\mathbf{s},[0,y]) \ \mathrm{d}\mu_{C^{1:d-1}}(\mathbf{s}) = \int_{[\mathbf{0},\mathbf{x}]\cap \Upsilon} K_C(\mathbf{s},[0,y]) c^{1:d-1}(\mathbf{s}) \ \mathrm{d}\lambda(\mathbf{s}) \\
		&= \int_{\Upsilon\cap \lbrace \mathbf{t}\in(0,1)^{d-1}: y \geq f^0(\mathbf{t}) \rbrace \cap [\mathbf{0},\mathbf{x}]} \tfrac{D^{-}\psi^{(d-2)}(\sum_{i=1}^{d-1}\varphi(s_i)+\varphi(y))}{D^{-}\psi^{(d-2)}(\sum_{i=1}^{d-1}\varphi(s_i))} \prod_{i=1}^{d-1}\varphi'(s_i) D^-\psi^{(d-2)}\left( \sum_{i=1}^{d-1}\varphi(s_i) \right) \ \mathrm{d}\lambda(\mathbf{s}) \\
		&= \int_{\lbrace\mathbf{t}\in(0,1)^{d-1}\setminus L_0^{1:d-1}: \,y\geq f^0(\mathbf{t})\rbrace \cap [\mathbf{0},\mathbf{x}]} D^{-}\psi^{(d-2)}\left(\sum_{i=1}^{d-1}\varphi(s_i)+\varphi(y)\right) \cdot \prod_{i=1}^{d-1}\varphi'(s_i) \ \mathrm{d}\lambda(\mathbf{s}).
	\end{align*}
	Notice that on the one hand, for $\mathbf{s}\not\in L_0^{1:d-1}$ we have $y< f^0(\mathbf{s})$ if, and only if, $\sum_{i = 1}^{d-1}\varphi(s_i) + \varphi(y) > \varphi(0)$, implying $D^-\psi^{(d-2)}\left( \sum_{i = 1}^{d-1}\varphi(s_i) + \varphi(y) \right) = 0$, and, on the other hand, $\mathbf{s}\in L_0^{1:d-1}$ also yields 
	$D^-\psi^{(d-2)}\left( \sum_{i = 1}^{d-1}\varphi(s_i) + \varphi(y) \right) = 0$. 
	Therefore in the strict and the non-strict case we get
	\begin{align*}
		I &= \int_{(\mathbf{0},\mathbf{x}] \setminus L_0^{1:d-1}} D^{-}\psi^{(d-2)}\left(\sum_{i=1}^{d-1}\varphi(s_i)+\varphi(y)\right) \cdot \prod_{i=1}^{d-1}\varphi'(s_i) \ \mathrm{d}\lambda(\mathbf{s}) \\
		&= \int_{(\mathbf{0},\mathbf{x}]\cap(0,1)^{d-1}} D^{-}\psi^{(d-2)}\left(\sum_{i=1}^{d-1}\varphi(s_i)+\varphi(y)\right) \cdot \prod_{i=1}^{d-1}\varphi'(s_i) \ \mathrm{d}\lambda(\mathbf{s}) \\
		&= \int_{(\mathbf{0},\mathbf{x}_{1:d-2}] }\prod_{i=1}^{d-2}\varphi'(s_i) \int_{(0,x_{d-1}]} \varphi'(s_{d-2}) D^-\psi^{(d-2)}\left( \sum_{i=1}^{d-1}\varphi(s_i)+\varphi(y) \right) \mathrm{d}\lambda(s_{d-1})\mathrm{d}\lambda(\mathbf{s}_{1:d-2}).
	\end{align*}
	Hence, using change of coordinates together with the fact that $\psi^{(m)}(\infty) = 0$ holds 
	for $m=1,2,\ldots, d-2$ it follows that 
	\begin{align*}
		I &= \int_{(\mathbf{0},\mathbf{x}_{1:d-2}] }\prod_{i=1}^{d-2}\varphi'(s_i) \lim_{\Delta\to 0} \Biggl\{ \psi^{(d-2)}\left( \sum_{i = 1}^{d-2}\varphi(s_i) + \varphi(x_{d-1}) + \varphi(y) \right) \\
		&\qquad - \psi^{(d-2)}\left( \sum_{i = 1}^{d-2}\varphi(s_i) + \varphi(\Delta) + \varphi(y) \right) \Biggr\} \mathrm{d}\lambda(\mathbf{s}_{1:d-2}) \\
		&= \int_{(\mathbf{0},\mathbf{x}_{1:d-2}] }\prod_{i=1}^{d-2}\varphi'(s_i) \psi^{(d-2)}\left( \sum_{i = 1}^{d-2}\varphi(s_i) + \varphi(x_{d-1}) + \varphi(y) \right) \mathrm{d}\lambda(\mathbf{s}_{1:d-2}).
	\end{align*}
	Proceeding analogously $d-3$ times finally yields 
	\begin{align*}
		I &= \int_{(0,x_1]}\varphi'(s_1) \psi'\left( \varphi(s_1) + \varphi(x_2)+\cdots+\varphi(x_{d-1}) + \varphi(y) \right) \mathrm{d}\lambda(s_1) \\
		&= \psi\left( \sum_{i = 1}^{d-1}\varphi(x_i) + \varphi(y) \right) = C(\mathbf{x},y)
	\end{align*}
	as desired.
\end{proof}

\begin{Rem}
As mentioned in the proof above $\mu_{C^{1:d-1}}(L_0^{1:d-1}) = 0$, implying that in the first line in equation \eqref{eq:mk1} we could 
substitute the univariate distribution function $y \mapsto 1$ by any other univariate distribution function $F$ 
fulfilling $F(1)=1$. 
\end{Rem}

\begin{Rem}
It also seems feasible to consider $m$-kernels for some $m \in \{2,\ldots,d-2\}$ instead of $(d-1)$-kernels, i.e.,
to condition on $m$ instead of $d-1$ coordinates. As shown by the subsequent results, however, although 
$d-1$-kernels involve the highest derivatives of the generator they are easy to handle and provide various new results,
particularly with respect to the interplay with the Williamson transform as discussed in Section \ref{sec5:williamson}.
\end{Rem}

Utilizing Theorem \ref{th:multi:kernel} allows to derive the well-known formulas (\cite{multiArchNeslehova}) 
for the level set masses and the 
Kendall distribution function of multivariate Archimedean copulas easily via Markov kernels and disintegration
(see Propositions \ref{prop:level:surfaces} and \ref{prop:multivariate:kendall} in the Appendix). 
In fact, for $t\in(0,1]$ the identity 
\begin{eqnarray}\label{eq:levelset:mass}
	\mu_C(L_t) = \frac{(-\varphi(t))^{d-1}}{(d-1)!} \cdot \big( D^-\psi^{(d-2)}(\varphi(t)) - D^-\psi^{(d-2)}(\varphi(t-)) \big).
\end{eqnarray}
can be shown. Moreover, if $C$ is strict then $\mu_C(L_0) = 0$ and for non-strict $C$ 
\begin{eqnarray}\label{eq:levelset:mass_zero}
	\mu_C(L_0) &= \frac{(-\varphi(0))^{d-1}}{(d-1)!} \cdot D^-\psi^{(d-2)}(\varphi(0))
\end{eqnarray}
holds. Letting $F_K^d$ denote the Kendall distribution function of $C$, for $t>0$ we have
\begin{eqnarray}\label{eq:kendall.series}
	F_K^d(t) = D^-\psi^{(d-2)}(\varphi(t))  \frac{(-1)^{d-1}}{(d-1)!} \varphi(t)^{d-1} + \sum_{k=0}^{d-2} \psi^{(k)}(\varphi(t)) \frac{(-1)^k}{k!}\varphi(t)^k.
\end{eqnarray}
Moreover for $t=0$ and strict $C$ we have $F_K^d(0) = 0$, and for non-strict $C$
\begin{eqnarray}
	F_K^d(0) = D^-\psi^{(d-2)}(\varphi(0)) \cdot \frac{(-1)^{d-1}}{(d-1)!}\cdot \varphi(0)^{d-1}.
\end{eqnarray}
holds.  

\begin{Rem}
	In the bivariate setting it is well-known that the formula for the level set masses can nicely be 
	interpreted geometrically (see \cite{p21, Nelsen}). 
	In fact, following \cite{p21}, given a discontinuity $b$ of  $D^+\varphi$ we have
	\begin{align*}
		\mu_C(L_b) = \varphi(b) \cdot \left( \frac{1}{D^-\varphi(b)} - \frac{1}{D^+\varphi(b)} \right),
	\end{align*}
	i.e., $\mu_C(L_b)$ corresponds to the length of the line segment on the $x$-axis generated by the left-hand and right-hand tangents of $\varphi$ at $b$ (see Figure 1 in \cite{p21}). 
	Translating from $\varphi$ and $x$-axis to $\psi$ and $y$-axis yields as special case of eq. \eqref{eq:levelset:mass} 
	\begin{align*}
		\mu_C(L_b) = \varphi(b)\cdot\left( D^-\psi(\varphi(b)) - D^-\psi(\varphi(b-)) \right).
	\end{align*}
	
	To establish the multivariate geometric analogue, rather than tangent lines we can 
	consider the left and right hand Taylor polynomials of $\psi$ of order $d-1$ at $a\in[0,\infty)$, i.e., 
	\begin{align*}
		T_{d-1}^\pm\psi(z, a) := D^\pm\psi^{(d-2)}(a) \cdot \frac{(z-a)^{d-1}}{(d-1)!} + \sum_{k=0}^{d-2} \psi^{(k)}(a) \cdot \frac{(z-a)^k}{k!}.
	\end{align*}
	Having that yields 
	$$
	T^{-}_{d-1}\psi(0,\varphi(t)) = D^{-}\psi^{(d-2)}(\varphi(t))\frac{(-1)^{d-1}\varphi(t)^{d-1}}{(d-1)!} + \sum_{k=0}^{d-2}\psi^{(k)}(\varphi(t))\frac{(-1)^k\varphi(t)^k}{k!} = F_K^d(t)
	$$
	and \eqref{eq:levelset:mass} reduces to
	\begin{align*}
		\mu_C(L_t) = T_{d-1}^-\psi(0, \varphi(t)) - T_{d-1}^+\psi(0, \varphi(t)) = F_K^d(t) - F_K^d(t-),
	\end{align*}
	i.e., the $t$-level set mass corresponds to the difference of $y$-intercepts of the left and right hand 
	Taylor polynomials of $\psi$ of order $d-1$ at $\varphi(t)$.
	Panel 2 of Figure \ref{fig:levelmass:taylor} in Example \ref{ex:taylor} illustrates this interpretation for $d=3$.
\end{Rem}

\begin{Ex}\label{ex:taylor}
	We construct an Archimedean generator $\psi$ inducing some $C_\psi\in\mathcal{C}_\text{ar}^3$ with 
	$1=z_0$ being a discontinuity point of $(D^-\psi')$ and provide a geometric interpretation of $\mu_C(L_{z_0})$ 
	in terms of $\psi$. 
	As depicted in the left panel of Figure \ref{fig:levelmass:taylor} we start with some non-negative, convex and 
	decreasing function $(-1)\tilde{\psi}^{(1)}$ (gray) with 
	$D^+\tilde{\psi}^{(1)}(z_0) \neq D^-\tilde{\psi}^{(1)}(z_0)$. 
	Considering 
	\begin{align*}
		\tilde{\psi}(z) := \int_{[z,\infty)} (-1)\tilde{\psi}^{(1)}(s) \ \mathrm{d}\lambda(s)
	\end{align*}
	as well as setting  $\psi(z) = \tilde{\psi}(z) / \tilde{\psi}(0)$ yields a generator of a 
	three-dimensional Archimedean copula $C_\psi$ (gray curve in the right panel).  
	The length of the vertical red line segment formed by the left and right hand Taylor polynomials of order $2$ at 
	$z_0$ (right panel) coincides with $\mu_C(L_{z_0})$.
	\begin{figure}[!htp]
		\centering
		\includegraphics[width=\textwidth]{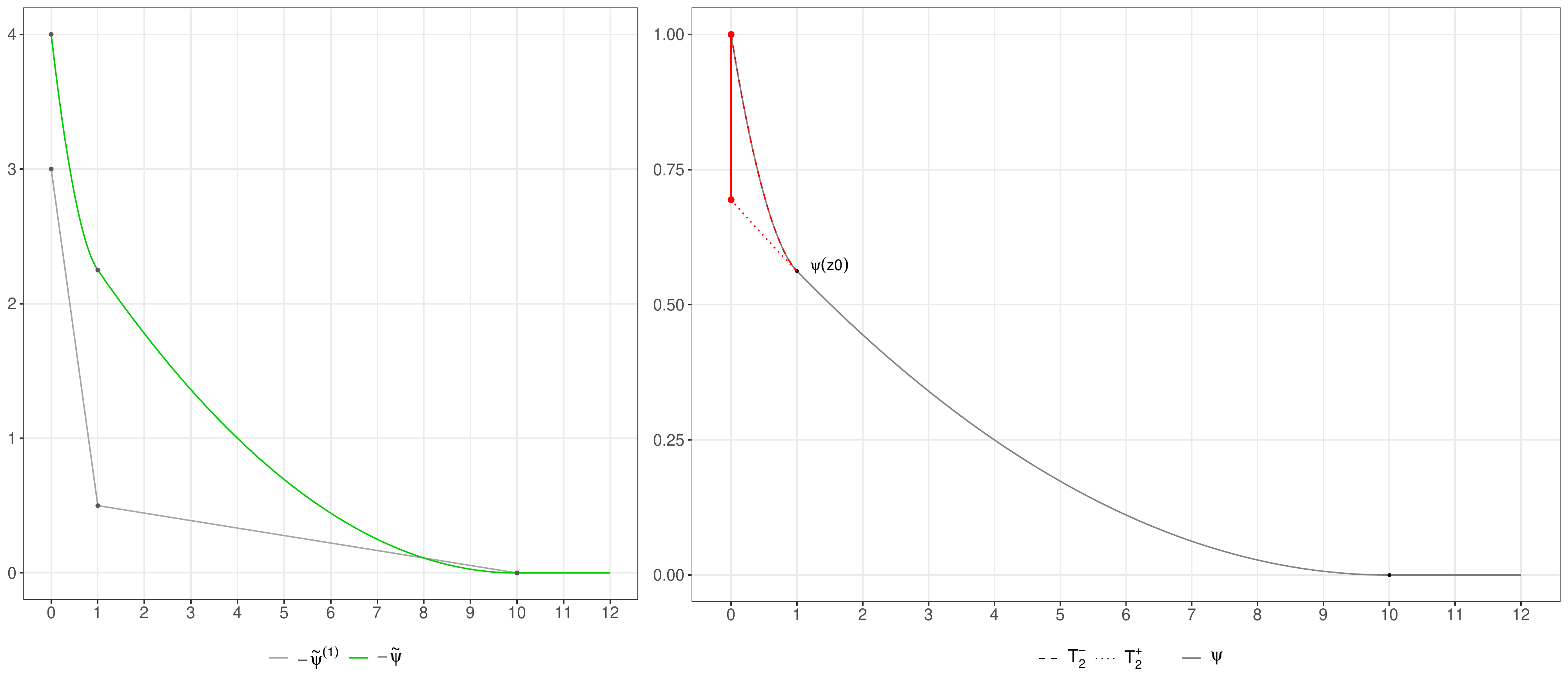}
		\caption{Stepwise construction of a non-strict $3$-monotone Archimedean generator and geometric interpretation 
		of the $z_0$-level set mass in terms of the vertical red line segment formed by the 
		two Taylor parabolas in $z_0$ as considered in Example \ref{ex:taylor}.}
		\label{fig:levelmass:taylor}
	\end{figure} 
\end{Ex}

\section{Characterizing pointwise convergence in $\mathcal{C}^d_{ar}$ and the interrelation with weak conditional 
convergence}
\label{sec:convergence}
In \cite[Section 6]{convArch} it was shown that within the class of bivariate Archimedean copulas pointwise 
convergence and weak conditional convergence (defined as weak conditional convergence of $\lambda$-almost all
conditional distributions) are equivalent. Moreover, it was shown that in the general bivariate setting, 
weak conditional convergence implies pointwise convergence not necessarily vice versa. 
We now tackle the question whether analogous results hold in the general $d\geq 3$-dimensional setting. Considering that - 
contrary to the bivariate case $d=2$ - for two copulas $A,B \in \mathcal{C}^d $ we do not necessarily have
$A^{1:d-1}=B^{1:d-1}$ we first need to discuss potential extensions of the notion of weak conditional convergence 
to $\mathcal{C}^d$.     

The seemingly most natural approach would be to say that a sequence $(A_n)_{n\in\mathbb{N}}$ of $d$-dimensional copulas  converges \emph{weakly conditional} to $A\in\mathcal{C}^d$ if, and only if, there exists a set 
$\Lambda$ with $\mu_{A^{1:d-1}}(\Lambda) = 1$ such that for every $\mathbf{x}\in\Lambda$ the sequence $\left(K_{A_n}(\mathbf{x}, \cdot)\right)_{n\in\mathbb{N}}$ of probability measures on $\mathcal{B}(\mathbb{I})$ converges to the probability measure $K_A(\mathbf{x}, \cdot)$.
As pointed out in \cite{qmd}, however, for $A,B \in \mathcal{C}^d$ we might even have that 
$A^{1:d-1}, B^{1:d-1}$ (or, more precisely, the measures $\mu_{A^{1:d-1}}, \mu_{B^{1:d-1}}$) 
are singular with respect to each other, making it unreasonable to compare $K_A(\mathbf{x},\cdot)$ and 
$K_B(\mathbf{x}, \cdot)$ since they are only defined uniquely $\mu_{A^{1:d-1}}$-a.e. and $\mu_{B^{1:d-1}}$-a.e., respectively (see \cite[Example 4.10]{qmd} for an illustration of this scenario).

As a consequence, the afore-mentioned natural concept of weak conditional convergence does not yield a reasonable notion in full generality. For certain families of copulas such as classes of copulas with identical marginals, or for 
Archimedean copulas, however, considering weak conditional convergence does make sense. 
In fact, given $C_1,C_2\in\mathcal{C}_\text{ar}^d$ we already know that $C_1^{1:d-1}$ and $C_2^{1:d-1}$ are both 
not only absolutely continuous but according to eq. \eqref{eq:arch:density} the corresponding densities $c_1^{1:d-1}, c_2^{1:d-1}$ fulfill $c_i^{1:d-1}(\mathbf{x})>0$ for every $\mathbf{x}$ outside of the respective zero set $L_0^{1:d-1}$, so it can not happen that $C_1^{1:d-1}$ and $C_2^{1:d-1}$ are singular w.r.t. each other.

In what follows we will therefore work with the afore-mentioned notion of weak conditional convergence
in $\mathcal{C}^d_{ar}$ and prove the following main result of this section in several steps: 

\begin{theorem}\label{th:main}
	Suppose that $C,C_1,C_2,\ldots$ are $d$-dimensional Archimedean copulas with generators $\psi, \psi_1, \psi_2,\ldots$, respectively. Then the following assertions are equivalent (\emph{Cont}$(g)$ denotes the set of continuity points of a 
	function $g$):
	\begin{enumerate}
		\item $(C_n)_{n\in\mathbb{N}}$ converges uniformly to $C$.
		\item $(C_n^{i:j})_{n\in\mathbb{N}}$ converges uniformly to $C^{i:j}$ for all $i,j \in \{1,\ldots, d\}$
		with $i \neq j$.
		\item $(\varphi_n)_{n\in\mathbb{N}}$ converges pointwise to $\varphi$ on $(0,1]$.
		\item $(\psi_n)_{n\in\mathbb{N}}$ converges uniformly to $\psi$ on $[0,\infty)$.
		\item $(\psi_n^{(m)})_{n\in\mathbb{N}}$ converges pointwise to $\psi^{(m)}$ on $(0,\infty)$ for 
		every $m\in\{1,2,\ldots, d-2\}$ and 
		$(D^-\psi_n^{(d-2)})_{n\in\mathbb{N}}$ converges pointwise to $D^-\psi^{(d-2)}$ on 
		\emph{Cont}$(D^-\psi^{(d-2)})$.
		\item $(c_n^{1:m})_{n\in\mathbb{N}}$ converges to $c^{1:m}$ $\lambda_d$-almost everywhere in 
		 $\mathbb{I}^{m}$ for every $m\in\{2,3,\ldots, d-1\}$.
	\end{enumerate}
	Furthermore, any of the six assertions implies weak conditional convergence of $(C_n)_{n\in\mathbb{N}}$ to $C$.
\end{theorem}
Obviously, assuming uniform convergence of $(C_n)_{n\in\mathbb{N}}$ to $C$ yields convergence of all marginal copulas and, in particular, of the bivariate marginals. Thus, the equivalence of (1), (2) and (3) follows directly from the 
results in the two-dimensional setting established in \cite{wcc}. 

\begin{Rem}
	Theorem \ref{th:main} constitutes the natural extension of \cite[Theorem 4.2]{wcc} characterizing
	 uniform convergence within the space of bivariate Archimedean copulas. In the multivariate setting, however, 
	 the afore-mentioned notion of weak conditional convergence in $\mathcal{C}^d_{ar}$ 
	 is a consequence of rather than an equivalence to uniform convergence. Slightly modifying 
	 the notion and incorporating the marginal densities, however, an equivalence can be established, see 
	 point (5) in Theorem \ref{rem:wcc.alternative} at the end of this section.
\end{Rem}

In what follows, we show the equivalence of the assertions (4), (5) and (6) and then conclude the section by deriving weak conditional convergence in several steps. We start with the following lemma clarifying the relationship between
 convergence of the generators and their pseudo-inverses. The proof of an analogous result in the context of 
 t-norms with multiplicative generators can be found in \cite[Theorem 8.14]{tnorms}.

\begin{Lemma}\label{lem:conv:psi}
	Suppose that $\psi, \psi_1, \psi_2,\ldots$ are Archimedean generators with pseudo-inverses 
	$\varphi, \varphi_1, \varphi_2,\ldots$, respectively. Then $(\varphi_n)_{n\in\mathbb{N}}$ converges pointwise to 
	$\varphi$ on $(0,1]$ if, and only if, $(\psi_n)_{n\in\mathbb{N}}$ converges to $\psi$ uniformly on $[0,\infty)$.
\end{Lemma}
\begin{proof}
	Defining 
	\begin{align*}
		\tilde{\psi}(w) := \begin{cases}
			\psi(-w) & \text{ if } w\leq 0 \\
			1 & \text{ otherwise}
		\end{cases}
	\end{align*}
	yields a univariate distribution function $\tilde{\psi}$ on $\mathbb{R}$ whose pseudo-inverse 
	$(\tilde{\psi})^-: (0,1] 
	\rightarrow \mathbb{R}$ coincides with $-\varphi$. Considering that weak convergence of distribution functions is equivalent to weak convergence of their pseudo-inverses (see \cite[Lemma 21.2]{vaart_1998}),  
	 using continuity of the involved functions it directly follows
	  that pointwise convergence of $\tilde{\psi}_n$ to $\tilde{\psi}$ pointwise on $\mathbb{R}$ is equivalent to pointwise
	  convergence of $-\varphi_n$ to $-\varphi$ on $(0,1]$ for $n\to\infty$. 
	  In other words: $\psi_n\to\psi$ pointwise on $[0,\infty)$ if, and only if, $\varphi_n\to\varphi$ pointwise 
	  on $(0,1]$. Finally, uniform convergence of $\tilde{\psi}_n$ to $\tilde{\psi}$ on $\mathbb{R}$ 
	  is a direct consequence of the fact that pointwise convergence of a sequence of univariate distribution 
	  functions to a continuous distribution function implies uniform convergence. 
	  Having that, obviously uniform convergence of $(\psi_n)_{n\in\mathbb{N}}$ to $\psi$ on $[0,\infty)$ follows.
\end{proof}

The next lemma focuses on convergence of the derivatives of the Archimedean genera\-tors. Thereby we say that 
a sequence $(f_n)_{n\in\mathbb{N}}$ of real functions $(f_n)_{n\in\mathbb{N}}$ 
\emph{converges continuously} to $f$ if for every sequence $(x_n)_{n\in\mathbb{N}}$ converging to $x \in$ Cont$(f)$ 
we have $\lim_{n\to\infty} f_n(x_n) = f(x)$ (cf. \cite{kosmol}).

\begin{Lemma}\label{lem:conv:derivatives}
	Suppose that $C,C_1,C_2,\ldots$ are $d$-dimensional Archimedean copulas with 
	generators $\psi, \psi_1, \psi_2,\ldots$, respectively. 
	If $(C_n)_{n\in\mathbb{N}}$ converges pointwise to $C$ then for every $m\in\{0,1,2,\ldots, d-2\}$ we have 
	\begin{align*}
		\lim_{n\to\infty} \psi^{(m)}_n(z) = \psi^{(m)}(z).
	\end{align*}
	for every $z\in(0,\infty)$ as well as $\lim_{n\to\infty} D^-\psi^{(d-2)}_n(z) = D^-\psi^{(d-2)}(z)$ 
	for every $z\in$ \emph{Cont}$(D^-\psi^{(d-2)})$. Moreover, in both situations the convergence is continuous, i.e., 
	 for every sequence $(z_n)_{n\in \mathbb{N}}$ in $(0,\infty)$ converging to 
	  $z \in \emph{Cont}(D^-\psi^{(d-2)})$ it holds that 
	  $\lim_{n\rightarrow \infty}D^-\psi_n^{(d-2)}(z_n) = D^-\psi^{(d-2)}(z)$.\\
	Vice versa, if $(\psi_n^{(m)})_{n\in\mathbb{N}}$ converges pointwise to $\psi^{(m)}$ on $(0,\infty)$ for 
	some \mbox{$m\in\{1,2,\ldots,d-2\}$} then $(C_n)_{n\in\mathbb{N}}$ converges pointwise to $C$ and the same 
	implication holds if $\lim_{n\to\infty} D^-\psi^{(d-2)}_n(z) = D^-\psi^{(d-2)}(z)$ for every $z\in$ 
	\emph{Cont}$(D^-\psi^{(d-2)})$.
\end{Lemma} 
\begin{proof}
	 We already know that pointwise convergence of $(C_n)_{n\in\mathbb{N}}$ to $C$ is equivalent to 
	 uniform convergence of $(\psi_n)_{n\in\mathbb{N}}$ to $\psi$ on $[0,\infty)$. 
	 Convexity of generators implies that the sequence $(\psi_n')_{n\in\mathbb{N}}$ converges pointwise to $\psi'$ on 
	 Cont$(\psi') = (0,\infty)$ (see, e.g., \cite{rockafellar}). Consi\-dering that 
	 $(-1)^{m}\psi^{(m)},(-1)^{m}\psi_1^{(m)},(-1)^{m}\psi_2^{(m)},\ldots$ are convex functions for every 
	 $m \in \{1,\ldots,d-2\}$ too, applying \cite{rockafellar}, continuous convergence follows. The fact that 
	 $(D^-\psi_n^{(d-2)})_{n\in\mathbb{N}}$ converges continuously to $D^-\psi^{(d-2)}$ on 
	 Cont$(D^-\psi^{(d-2)})$ is a consequence of Lemma \ref{convex.cont.conv} in the Appendix. \\	  
	We prove the reverse implication only for the case $d=3$ and assume that 
	$(D^-\psi_n^{(d-2)})_{n\in\mathbb{N}}$ converges to $D^-\psi^{(d-2)}$ on Cont$(D^-\psi^{(d-2)})$ 
	since the analogous 	statement for $m \in \{1,\ldots,d-2\}$ follows in the same manner and the extension to 
	arbitrary $d \geq 3$ is obvious. 
	For $z \in [0,\infty)$ according to eq. (\ref{eq:up}) we have 
	\begin{align*}
	  \psi(z) &= \int_{[z,\infty)} -\psi'(s) \mathrm{d}\lambda(s) = 
	    \int_{[z,\infty)} -\int_{[s,\infty)} -D^-\psi'(t) \mathrm{d}\lambda(t) \mathrm{d}\lambda(s) \\
	          &= \int_{[z,\infty)} \int_{[s,\infty)} \underbrace{D^-\psi'(t)}_{\geq 0} \mathrm{d}\lambda(t) 
	            \mathrm{d}\lambda(s)
	\end{align*}
	and the same holds for every $\psi_n$. Using $\psi(0)=\psi_n(0)=1$ we can interpret the functions
	$\iota_n, \iota: \Delta \rightarrow [0,\infty)$, defined by 
     \begin{align*}
       \iota_n(t,s) := D^-\psi_n'(t), \quad \iota(t,s) := D^-\psi'(t)
\end{align*}     	
    	as probability densities on the measure space $(\Delta,\mathcal{B}(\Delta),\lambda_2)$ with $\Delta$ denoting the 
    	closed set $\Delta:=\{(x,y) \in [0,\infty)^2: y \geq x\}$. 
    By assumption, the sequence $(\iota_n)_{n \in \mathbb{N}}$ converges $\lambda_2$-almost everywhere on $\Delta$ 
    to $\iota$, so applying Scheffe's theorem (or Riesz' theorem, see \cite{Kuso}) yields
    \begin{align*}
     \lim_{n \rightarrow \infty} \int_\Delta \vert \iota_n (t,s) - \iota(t,s) \vert d\lambda_2(t,s) =0.
    \end{align*}        	
    Hence, for an arbitrary $z \in (0, \infty)$ using the triangle inequality it follows that 
    \begin{align*}
    \vert \psi_n(z) - \psi(z) \vert &= \left \vert 
       \int_{[z,\infty)} \int_{[s,\infty]} D^-\psi_n'(t) - D^-\psi'(t) \, \mathrm{d}\lambda(t) \mathrm{d}\lambda(s)
       \right \vert \\
       & \leq \int_{\Delta \cap [z,\infty)^2} \vert \iota_n (t,s) - \iota(t,s) \vert \mathrm{d}\lambda_2(t,s) 
       \leq \int_\Delta \vert \iota_n - \iota \vert \mathrm{d}\lambda_2 \rightarrow 0
\end{align*}     
    for $n \rightarrow \infty$, which completes the proof. 
\end{proof}

Altogether we have already established the equivalence of the first five assertions in Theorem \ref{th:main}
and it remains to show the equivalence of the sixth assertion, which is tackled in the following lemma: 

\begin{Lemma}\label{lem:conv:densities}
	Suppose that $C,C_1,C_2,\ldots$ are $d$-dimensional Archimedean copulas with generators $\psi, \psi_1, \psi_2,\ldots$, respectively. Then pointwise convergence of $(C_n)_{n\in\mathbb{N}}$ to $C$ is equivalent to pointwise convergence of 
	$\left(c_n^{1:m}\right)_{n\in\mathbb{N}}$ to $c^{1:m}$ almost everywhere in $\mathbb{I}^{m}$ for some  
	$m\in\{2,3,\ldots, d-1\}$.
\end{Lemma}
\begin{proof}
	We only prove the equivalence for $m=d-1$ since considering the fact that 
	Cont$(\psi^{(m)}) = (0,\infty)$ holds for $m=1,2,\ldots,d-2$ 
	all other cases follow analogously. \\
	First observe that the set 
	\begin{align*}
		\Gamma := \left\lbrace\mathbf{x}\in(0,1)^{d-1}: \sum_{i=1}^{d-1} \varphi(x_i) \in \text{Cont}(D^-\psi^{(d-2)}) \right\rbrace \in \mathcal{B}(\mathbb{I}^{d-1})
	\end{align*}
	has full Lebesgue measure in $\mathbb{I}^{d-1}$. In fact, considering $\Gamma^c := (0,1)^{d-1} \setminus \Gamma$,
	applying disintegration and writing $\mathbf{x}_{1:d-2} = (x_1,x_2,\ldots,x_{d-2})$ yields
	\begin{align}\label{eq:temp.sec4}
		\lambda_{d-1}(\Gamma^c) = \int\limits_{(0,1)^{d-2}} \lambda((\Gamma^c)_{\mathbf{x}_{1:d-2}}) \ \mathrm{d}\lambda_{d-2}(\mathbf{x}_{d-2}).
	\end{align}
	For arbitrary $\mathbf{x}_{1:d-2} \in (0,1)^{d-2}$ obviously the $\mathbf{x}_{1:d-2}$-cut 
	$(\Gamma^c)_{\mathbf{x}_{1:d-2}}$ of $\Gamma^c$ fulfills 
   \begin{align*}
   (\Gamma^c)_{\mathbf{x}_{1:d-2}}= \left\lbrace x_{d-1}\in (0,1): \sum_{i=1}^{d-1} \varphi(x_i) \not \in 
   \text{Cont}(D^-\psi^{(d-2)}) \right\rbrace. 
\end{align*}   	
Considering the fact that $D^-\psi^{(d-2)}$ has at most countably infinitely many discontinuities and that
$\varphi$ is strictly decreasing it follows that $(\Gamma^c)_{\mathbf{x}_{1:d-2}}$ is at most countably infinite
and therefore has $\lambda$-measure $0$. Applying eq. (\ref{eq:temp.sec4}) therefore 
directly yields $\lambda_{d-1}(\Gamma^c)=0$, implying $\lambda_{d-1}(\Gamma)=1$.  \\	
In case $\lim_{n \rightarrow \infty} d_\infty(C_n,C)=0$ holds, for $\mathbf{x}\in \Gamma$ applying Lemma
 \ref{lem:conv:derivatives} yields 
	\begin{align*}
		\lim_{n\to\infty} c^{1:d-1}_n(\mathbf{x}) &= \lim_{n\to\infty} \prod_{i=1}^{d-1} \varphi_n'(x_i) \cdot D^-\psi_n^{(d-2)}\left( \sum_{i=1}^{d-1}\varphi_n(x_i) \right) \\ &= \prod_{i=1}^{d-1} \varphi'(x_i) \cdot D^-\psi^{(d-2)}\left( \sum_{i=1}^{d-1}\varphi(x_i) \right) = c^{1:d-1}(\mathbf{x}).
	\end{align*}
	Conversely, almost everywhere convergence of $(c^{1:m}_n)_{n \in \mathbb{N}}$ to $c^{1:m}$ implies 
	almost everywhere convergence of $(c^{1:2}_n)_{n \in \mathbb{N}}$ to $c^{1:2}$. Using Scheff\'e's theorem 
	we get that $(C^{1:2}_n)_{n \in \mathbb{N}}$ converges pointwise to $C^{1:2}$. Since convergence of bivariate 
	Archimedean copulas is equivalent to pointwise convergence of $(\varphi_n)_{n \in \mathbb{N}}$ to $\varphi$ 
	applying Lemma \ref{lem:conv:psi} yields uniform convergence of $(\psi_n)_{n \in \mathbb{N}}$ to $\psi$ 
	and the assertion follows from the already established equivalence of the first four assertions of 
	 Theorem \ref{th:main}. 
\end{proof}

\noindent Having proved the equivalence of the six conditions in Theorem \ref{th:main} we now show that any of these 
properties implies weak 
conditional convergence. Notice that for the case that the (Archimedean) limit copula $C$ is non-strict 
for every $z > \varphi(0)$ we obviously have 
\begin{align*}
		\lim_{n\to\infty} (-1)^{d-2}D^-\psi^{(d-2)}_n(z) = 0.
	\end{align*}
We now focus on `good' $\mathbf{x} \in \mathbb{I}^{d-1}$ and $y \geq f^0(\mathbf{x})$ and show convergence of the 
Markov kernels. Notice that the subsequent lemma already establishes weak conditional convergence for the case
of a strict $d$-dimensional limit copula $C$. 

\begin{Lemma}\label{lem:conv:strict}
	Suppose that $C,C_1,C_2,\ldots$ are $d$-dimensional Archimedean copulas with generators 
	$\psi, \psi_1, \psi_2,\ldots$ and corresponding Markov kernels $K_C, K_{C_1}, K_{C_2},\ldots$, respectively. 
	If $(C_n)_{n\in\mathbb{N}}$ converges uniformly to $C$ then there exists a set 
	$\Lambda\in\mathcal{B}(\mathbb{I}^{d-1})$ with $\mu_{C^{1:d-1}}(\Lambda) = 1$ such that for every  
	$\mathbf{x}\in\Lambda$ the following assertion holds: there is some set 
	$U^\mathbf{x} \subseteq [f^0(\mathbf{x}), 1]$ which is dense in $[f^0(\mathbf{x}), 1]$ and which fulfills that 
	for every $y \in U^\mathbf{x}$ 
	\begin{align*}
		\lim_{n\to\infty} K_{C_n}(\mathbf{x}, [0,y]) = K_C(\mathbf{x}, [0,y])
	\end{align*}
	holds. 
\end{Lemma}
\begin{proof}
	As shown in Lemma \ref{lem:conv:densities} the set 
	\begin{align*}
		\Gamma = \left\lbrace\mathbf{x}\in (0,1)^{d-1}: \sum_{i=1}^{d-1} \varphi(x_i) \in \text{Cont}(D^-\psi^{(d-2)}) \right\rbrace
	\end{align*}
	satisfies $\lambda_{d-1}(\Gamma) = 1$. Considering $
		\mu_{C^{1:d-1}}(\Gamma) = \int_\Gamma c^{1:d-1} \ \mathrm{d}\lambda_{d-1} = 1$ it follows that 
	  $\Lambda := \Gamma\setminus L_0^{1:d-1}$ fulfills $\mu_{C^{1:d-1}}(\Lambda)$=1.
	For fixed $\mathbf{x}\in\Lambda$ it follows by the same reasoning as in \ref{lem:conv:densities} that
	\begin{align*}
		U^\mathbf{x} := \left\lbrace y\in [f^0(\mathbf{x}), 1]: \sum_{i=1}^{d-1}\varphi(x_i) + \varphi(y) \in \text{Cont}(D^-\psi^{(d-2)}) \right\rbrace
	\end{align*}
	is of full $\lambda$-measure in $[f^0(\mathbf{x}), 1]$. 
	Therefore applying Lemma \ref{lem:conv:derivatives} directly yields 
	\begin{align*}
		\lim_{n\to\infty}K_{C_n}(\mathbf{x}, [0,y]) &= \lim_{n\to\infty} \frac{D^-\psi^{(d-2)}_n\left(\sum_{i=1}^{d-1} \varphi_n(x_i) + \varphi_n(y)\right)}{D^-\psi^{(d-2)}_n\left(\sum_{i=1}^{d-1} \varphi_n(x_i)\right)} \\ &= \frac{D^-\psi^{(d-2)}\left(\sum_{i=1}^{d-1} \varphi(x_i) + \varphi(y)\right)}{D^-\psi^{(d-2)}\left(\sum_{i=1}^{d-1} \varphi(x_i)\right)} = K_C(\mathbf{x},[0,y]).
	\end{align*}
\end{proof}

It remains to show that we also have convergence for `good' $\mathbf{x}$ and $y < f^0(\mathbf{x})$ (only relevant in 
the case of non-strict limit $C$.)

\begin{Lemma}\label{lem:conv:nonstrict}
	Suppose that $C,C_1,C_2,\ldots$ are $d$-dimensional Archimedean copulas with generators $\psi, \psi_1, \psi_2,\ldots$ and corresponding Markov kernels $K_C, K_{C_1}, K_{C_2},\ldots$, respectively. If $(C_n)_{n\in\mathbb{N}}$ converges uniformly to $C$ then there exists a set $\Lambda\in\mathcal{B}(\mathbb{I}^{d-1})$ with $\mu_{C^{1:d-1}}(\Lambda) = 1$ such that for all $\mathbf{x}\in\Lambda$ we have 
	\begin{align*}
		\lim_{n\to\infty} K_{C_n}(\mathbf{x}, [0,y]) = 0 = K_C(\mathbf{x}, [0,y])
	\end{align*}
	for every $y\in [0,f^0(\mathbf{x}))$.
\end{Lemma}
\begin{proof}
	Obviously it suffices to consider non-strict $C$. 
	Let $\Lambda$ be as in the proof of Lemma \ref{lem:conv:strict} and fix $\mathbf{x}\in\Lambda$ and 
	$y \in (0,f^0(\mathbf{x}))$. Then $\sum_{i=1}^{d-1}\varphi(x_i) + \varphi(y) > \varphi(0)$ and we have  
    \begin{align*}
    \lim_{n \rightarrow \infty} \sum_{i=1}^{d-1}\varphi_n(x_i) + \varphi_n(y) = 
    \sum_{i=1}^{d-1}\varphi(x_i) + \varphi(y)> \varphi(0).
\end{align*}    	 
	 Applying Lemma \ref{lem:conv:derivatives} (continuous convergence) yields both  
$$
\lim_{n \rightarrow \infty}D^- \psi^{(d-2)}_n\left(\sum\limits_{i=1}^{d-1}\varphi_n(x_i) + \varphi_n(y)\right) = 
D^- \psi^{(d-2)}\left(\sum_{i=1}^{d-1}\varphi(x_i) + \varphi(y)\right) = 0
$$
and, again using $\mathbf{x}\in\Lambda$,  
$$
\lim_{n\to\infty}D^-\psi_n^{(d-2)}\left(\sum_{i=1}^{d-1}\varphi_n(x_i)\right) = D^-\psi^{(d-2)}\left(\sum_{i=1}^{d-1}\varphi(x_i)\right) \neq 0.
$$
Having this, according to eq. \ref{eq:mk1} the desired identity $\lim_{n\to\infty}K_{C_n}(\mathbf{x},[0,y]) = 0$ follows.
Analogous arguments (or, simply using the fact that for every $d$-dimensional copula $C$ 
we have $K_C(\mathbf{x},\{0\})=0$ for $\mu_{C^{1:d-1}}$-almost every $\mathbf{x} \in \mathbb{I}^{d-1}$) 
also apply in the case $y=0$, so the proof is complete.
\end{proof}

\begin{Rem}
	The proof of Lemma \ref{lem:conv:nonstrict} is also applicable in dimension $d=2$ and therefore 
	provides an alternative simpler version of the rather involved approach followed in \cite{wcc}. 
\end{Rem}

Theorem \ref{th:main} has the following straightforward consequence (analogous to the bivariate statement considered 
Theorem 4.2 in \cite{wcc}):

\begin{theorem}\label{rem:wcc.alternative}
Suppose that $C,C_1,C_2,\ldots$ are $d$-dimensional Archimedean copulas with generators $\psi, \psi_1, \psi_2,\ldots$, respectively. Then the following assertions are equivalent:
	\begin{enumerate}
		\item $(C_n)_{n\in\mathbb{N}}$ converges uniformly to $C$.
		\item $(\varphi_n)_{n\in\mathbb{N}}$ converges pointwise to $\varphi$ on $(0,1]$.
		\item $(\psi_n)_{n\in\mathbb{N}}$ converges uniformly to $\psi$ on $[0,\infty)$.
		\item $(\psi_n^{(m)})_{n\in\mathbb{N}}$ converges pointwise to $\psi^{(m)}$ on $(0,\infty)$ for 
		every $m\in\{1,2,\ldots, d-2\}$ and 
		$(D^-\psi_n^{(d-2)})_{n\in\mathbb{N}}$ converges pointwise to $D^-\psi^{(d-2)}$ on 
		\emph{Cont}$(D^-\psi^{(d-2)})$.
		\item $(C_n)_{n\in\mathbb{N}}$ converges weakly conditional to $C$ and 
		$(c_n^{1:d-1})_{n\in\mathbb{N}}$ converges to $c^{1:d-1}$ $\lambda_d$-almost everywhere in 
		 $\mathbb{I}^{d-1}$.
	\end{enumerate}
\end{theorem}
\begin{proof}
We already know that the first four conditions are equivalent and that each of them implies the fifth assertion.
On the other hand, if the conditions in (5) hold, then we also have $\lambda_m$-almost everywhere 
convergence of $(c_n^{1:m})_{n\in\mathbb{N}}$ to $c^{1:m}$ for every $m \in \{2,\ldots,d-2\}$ and
applying Theorem \ref{th:main} completes the proof.  
\end{proof}

\section{Archimedean copulas and the Williamson transform}\label{sec5:williamson}
Following \cite{schilling2012}, every $d$-monotone function can be represented via the so-called 
Williamson transform of a unique probability measure on $[0,\infty)$. As a consequence, we may describe and handle  
$d$-dimensional Archimedean copulas $C=C_\psi$ in terms of their corresponding probability measure $\gamma$ on $[0,\infty)$. 
In what follows we first establish some complementary useful results describing the interrelation between generator 
$\psi$ and probability measures $\gamma$, express masses of level sets as well as the Kendall distribution function
handily in terms of $\gamma$, show that regularity properties of the corresponding measure $\gamma$ carry over 
to the Archimedean copula and prove the fact that pointwise convergence of Archimedean copulas $C_1,C_2,\ldots$ to 
an Archimedean copula $C$ is equivalent to weak convergence of the corresponding probability measures 
$\gamma_1,\gamma_2,\ldots$ to the probability measure $\gamma$. 
Based on these facts we then finally show that both, the family of absolutely continuous 
and the family of singular Archimedian copulas is dense in $(\mathcal{C}_{ar}^d, d_\infty)$. \\

We start with recalling the following result (see \cite{multiArchNeslehova} and \cite{schilling2012}) where we 
write $f_+^m$ for the $m$-th power of the positive part $f_+$ of a function $f$, i.e., $f_+^m:=(f_+)^m$: 
\begin{theorem}\label{thm:will_trans}
Let $\psi \colon [0,\infty) \rightarrow \mathbb{I}$ be a function and $d \geq 2$. Then the following two conditions 
are equivalent:
\begin{itemize}
    \item[(1)] $\psi$ is the generator of a $d$-dimensional Archimedean copula $C_\psi$.
    \item[(2)] There exists a unique probability measure $\gamma$ on $\mathcal{B}([0,\infty))$ with 
    $\gamma(\{0\}) = 0$ such that 
     \begin{equation}\label{eq:gamma.to.psi}
     \psi(z) = \int_{[0,\infty)}(1-tz)_+^{d-1} \mathrm{d}\gamma(t) =: (\mathcal{W}_d\gamma)(z),
     \end{equation}
holds for every $z>0$. In other words, $\psi$ is the Williamson transform $\mathcal{W}_d\gamma$ of $\gamma$.
\end{itemize}
\end{theorem}
\noindent  
Obviously our assumed normalization property $\psi(1) = \frac{1}{2}$ translates to
\begin{equation}\label{eq:cond_meas}
\int_{\mathbb{I}}(1-t)^{d-1} \mathrm{d}\gamma(t) = \frac{1}{2}.
\end{equation}
In what follows we therefore only consider measures $\gamma$ fulfilling eq. (\ref{eq:cond_meas}), let 
$\mathcal{P}_\mathcal{W}$ denote the family of all these measures, i.e., 
\begin{equation}
\mathcal{P}_\mathcal{W}=\left\{ \gamma \in \mathcal{P}([0,\infty)): \gamma(\{0\}) = 0
\textrm{ and } \int_{\mathbb{I}}(1-t)^{d-1} \mathrm{d}\gamma(t) = \frac{1}{2} \right\}
\end{equation}
and refer to $\mathcal{P}_{\mathcal{W}_d}$ as the family of all $d$-Williamson measures.

Again following \cite{multiArchNeslehova} next we derive an explicit representation for the cumulative distribution 
function of $\gamma$ in terms of the Archimedean generator $\psi$:
\begin{Lemma}\label{lem:one_dim_measure}
Let $\psi$ be the generator of a $d$-dimensional Archimedean copula and $\gamma \in \mathcal{P}_{\mathcal{W}_d}$
 its corresponding Williamson measure. Then
\begin{equation}\label{eq:williamson_meas}
\gamma([0,z]) = \sum_{k = 0}^{d-2}\frac{(-1)^k\psi^{(k)}(\tfrac{1}{z})}{k!} \frac{1}{z^{k}} + \frac{(-1)^{d-1}D^-\psi^{(d-2)}(\tfrac{1}{z})}{(d-1)!}\frac{1}{z^{d-1}}
\end{equation}
holds for every $z>0$.
\end{Lemma}
\begin{proof}
Defining $F \colon (0,\infty) \rightarrow \mathbb{I}$ by $F(z) := \gamma([0,z])=\gamma((0,z])$ according to \cite{williamson} 
\begin{equation}\label{eq:one_dim_meas_cont}
F(z) = \sum_{k = 0}^{d-1}\frac{(-1)^k\psi^{(k)}(\tfrac{1}{z})}{k!}\frac{1}{z^{k}}
\end{equation}
holds for every continuity point $z$ of $F$, i.e., for every continuity point $z$ of $D^-\psi^{(d-2)}$. 
Set $\text{Cont}(D^-\psi^{(d-2)})$ and define the function 
$G: (0,\infty) \rightarrow \mathbb{I}$ by
\begin{align*}
G(z) :=
\sum_{k = 0}^{d-2}\frac{(-1)^k\psi^{(k)}(\frac{1}{z})}{k!}\frac{1}{z^{k}}+ 
  \frac{(-1)^{(d-1)}D^-\psi^{(d-2)}(\frac{1}{z})}{(d-1)!}\frac{1}{z^{d-1}}.
\end{align*}
Then the distribution function $F$ and the right-continuous function $G$ coincide on $\text{Cont}(D^-\psi^{(d-2)})$ 
and since the latter is dense in $(0,\infty)$ equality $F=G$ follows and the proof is complete. 
\end{proof}

The following lemma expresses $D^{-}\psi^{(d-2)}$ (appearing both in the numerator and the denominator of the Markov kernel
$K_C(\cdot,\cdot)$ in Theorem \ref{th:multi:kernel}) in terms of $\gamma$ and will be useful in the sequel:
\begin{Lemma}\label{lem:Left_hand_G}
Let $\psi$ be the generator of a $d$-dimensional Archimedean copula and $\gamma \in \mathcal{P}_{\mathcal{W}_d}$ 
be the corresponding Williamson measure. Then  
$$
0\geq G(z) := (-1)^{d-2}D^{-}\psi^{(d-2)}(z) = -(d-1)!\int_{(0,\frac{1}{z}]}t^{d-1} \mathrm{d}\gamma(t)
$$
holds for every $z>0$.
\end{Lemma}
\begin{proof}
According to \cite{schilling2012} we already know that for every $z>0$ we have
$$
(-1)^{d-2}\psi^{d-2}(z) = (d-1)!\int_{[0,\infty)}t^{d-2}(1-zt)_+ \mathrm{d}\gamma(t) =
(d-1)!\int_{(0,\infty)}t^{d-2}(1-zt)_+ \mathrm{d}\gamma(t).
$$
Since $(-1)^{d-2}\psi^{d-2}$ is convex, the left hand derivative exists everywhere in $(0,\infty)$ and is 
left-continuous. For fixed $z >0$ considering the left-hand difference quotient 
$$
Q_h(z):=(d-1)! \, \int_{(0,\infty)} t^{d-2}\frac{(1-t(z+h))_+ - (1-tz)_+}{h} \mathrm{d}\gamma(t)
$$
for $h \in (-z,0)$ we get 
\begin{align*}
Q_h(z) &=
\frac{(d-1)!}{h}\int_{(\frac{1}{z},\frac{1}{z+h})}t^{d-2}(1-t(z+h)) \mathrm{d}\gamma(t) \\& \qquad +
\frac{(d-1)!}{h} \int_{(0,\frac{1}{z}]} t^{d-2}[(1-t(z+h))-(1-tz)] \mathrm{d}\gamma(t) \\& =
\underbrace{\frac{(d-1)!}{h} \int_{(\frac{1}{z},\frac{1}{z+h})} t^{d-2}(1-t(z+h)) \mathrm{d}\gamma(t)}_{=:I_h}  
\,\, - (d-1)! \, \int_{[0,\frac{1}{z}]}t^{d-1} \mathrm{d}\gamma(t),
\end{align*}
and it suffices to show that $I_h$ converges to $0$. Using the monotonicity 
and non-negativity of the functions $t \mapsto t^{d-2}$ and $t \mapsto 1-t(z+h)$ on $(\frac{1}{z},\frac{1}{z+h})$
it follows that 
\begin{align*}
|I_h| &\leq \frac{1}{|h|} \frac{1}{(z+h)^{d-2}}\left(1-\frac{z+h}{z}\right)
   \gamma\left(\left(\frac{1}{z},\frac{1}{z+h}\right)\right) \\& =
\frac{1}{(z+h)^{d-2} \, z} \, \, \gamma\left(\left(\frac{1}{z},\frac{1}{z+h}\right)\right) 
\end{align*}
from which the assertion follows immediately. 
\end{proof}
As first application of the previous lemma we characterize strictness in terms of the Williamson measure $\gamma$:
\begin{Lemma}\label{lem:measure.strict.cop}
Suppose that $C$ is an Archimedean copula  with Williamson measure $\gamma \in \mathcal{P}_{\mathcal{W}_d}$.
Then $C$ is strict if, and only if
the support of $\gamma$ contains $0$, i.e., if $\gamma([0,r))>0$ for every $r>0$.
\end{Lemma}
\begin{proof}
If the support of $\gamma$ contains $0$ then obviously $\int_{(0,r)}t^{d-1} \mathrm{d}\gamma(t) >0$
for every $r>0$, hence applying Lemma \ref{lem:Left_hand_G} directly yields 
$(-1)^{d-1}D^{-}\psi^{(d-2)}(z) > 0$ for every $z >0$. Having this, strictness of $\psi$ follows immediately. \\
On the other hand, if there exists some $r>0$ with $\gamma([0,r))=0$ then considering $z_0=\frac{1}{2r}$
and again using Lemma \ref{lem:Left_hand_G} we have $(-1)^{d-1}D^{-}\psi^{(d-2)}(z_0)=0$. Since the function
$z \mapsto (-1)^{d-1}D^{-}\psi^{(d-2)}(z)$ is non-negative, non-increasing and left-continuous it follows that   
$(-1)^{d-1}D^{-}\psi^{(d-2)}(z)=0$ holds for every $z \geq z_0$. Hence $\psi(z_0)=0$ and $\psi$ is non-strict.  
\end{proof}

We now return to the formulas for the level set masses and the Kendall distribution function of Archimedean copulas 
as already mentioned in Section \ref{sec:3:mass:distri} and reformulate them elegantly in terms of the 
Williamson measure $\gamma$. 
Although surprising, to the best of the authors' knowledge these formulas seem to be new: 
\begin{theorem}\label{thm:level.sets.williamson}
Let $C$ be a $d$-dimensional Archimedean copula with generator $\psi$ and Williamson measure $\gamma$.
Then (compare with equations (\ref{eq:levelset:mass}) - (\ref{eq:kendall.series})):
\begin{align}\label{eq:L_t.via.gamma}
\mu_C(L_t) = \gamma(\{\tfrac{1}{\varphi(t)}\}),\, t \in (0,1] 
\end{align}
holds for every $t \in (0,1]$ and every $C$. Furthermore, for strict $C$ we have $\mu_C(L_0)=0$, and
 for non-strict $C$     
\begin{align}
\mu_C(L_0) = \gamma(\{\tfrac{1}{\varphi(0)}\}).
\end{align}
holds. Finally, the Kendall distribution function $F_K^d$ of $C$ fulfills 
\begin{align}\label{eq:kendall.gamma}
F_K^d(t) = \gamma([0,\tfrac{1}{\varphi(t)}])
\end{align}
for every $t \in (0,1]$.
\end{theorem}
\begin{proof}
First of all notice that the expression for the Kendall distribution function follows immediately from 
equation \eqref{eq:williamson_meas}. Furthermore for $t \in (0,1]$ considering $$\mu_C(L_t)=F_K^d(t)-F_K^d(t-)$$
eq. \eqref{eq:L_t.via.gamma} follows immediately from eq. \eqref{eq:kendall.gamma}.
Finally, using eq. \eqref{eq:levelset:mass_zero} and incorporating Lemma \ref{lem:Left_hand_G} yields
\begin{align*}
    \mu_C(L_0) = \frac{(-1)^{d-1}(\varphi(0))^{d-1}}{(d-1)!}D^-\psi^{(d-2)}(\varphi(0)) = \varphi(0)^{d-1}\int_{(0,\frac{1}{\varphi(0)}]}t^{d-1}\mathrm{d}\gamma(t).
\end{align*}
Since eq. (\ref{eq:gamma.to.psi}) implies that for every $z_0>0$ we have that $\psi(z_0)=0$ is equivalent to 
$\gamma((0,\tfrac{1}{z_0}))=0$, the right-hand side of the last equation simplifies to 
$\gamma(\{\tfrac{1}{\varphi(0)}\})$ and the proof is complete. 
\end{proof}

\begin{Rem}
Notice that eq. (\ref{eq:kendall.gamma}) implies $F_K^d(\frac{1}{2}) = \gamma([0,1])$.
More importantly, the (probably most) famous conjecture in the context of copulas, saying that 
for every fixed $d\geq 3$ two Archimedean copulas $C,D \in \mathcal{C}^d_{ar}$ are identical if, and only if
their Kendall distribution functions coincide (see \cite{GNZ} and \cite{GeRi}) 
would follow if it could be shown that the mapping
assigning each Williamson measure $\gamma$ the function $F_\gamma: \mathbb{I} \rightarrow \mathbb{I}$, defined by
$$
F_\gamma(t)=\gamma\left(\left[0,\frac{1}{\varphi_\gamma(t)}\right]\right)
$$
is injective, where $\varphi_\gamma$ denotes the pseudo-inverse of the generator $\psi=\mathcal{W}_d \, \gamma$.
\end{Rem}

\begin{Ex}\label{ex:level:set}
The probability measure $\gamma=\frac{7}{8} \, \delta_{1/4} + \frac{1}{8} \, \delta_{3/4}$ obviously 
fulfills $\gamma \in \mathcal{P}_{\mathcal{W}_3}$. The induced generator $\psi$ is given by  
\begin{align*}
		\psi(z) := \begin{cases}
			 1- \frac{5z}{8} + \frac{z^2}{8} & \text{ if } z < \frac{4}{3} \\
			 \frac{7}{8}\left(1-\frac{z}{4}\right)^2 & \text{ if } z \in \left[\frac{4}{3},4 \right] \\
			0 & \text{ otherwise}
		\end{cases}
	\end{align*} 
and it is straightforward to verify that $\varphi(\frac{7}{18})=\frac{4}{3}$ and $\varphi(0)=4$ holds.
Using Theorem \ref{thm:level.sets.williamson} therefore yields $\mu_C(L_0)=\frac{7}{8}$ as well as 
$\mu_C(L_{7/18})=\frac{1}{8}$. Figure \ref{fig:Fgamma:psi} depicts the distribution function $z \mapsto \gamma([0,z])$ 
of $\gamma$ (left panel), the induced generator $\psi$ (middle) and the sets $L_0$ and $L_{7/18}$	carrying the mass 
(right panel).
\begin{figure}[!htp]
		\centering
		\includegraphics[width=1\textwidth]{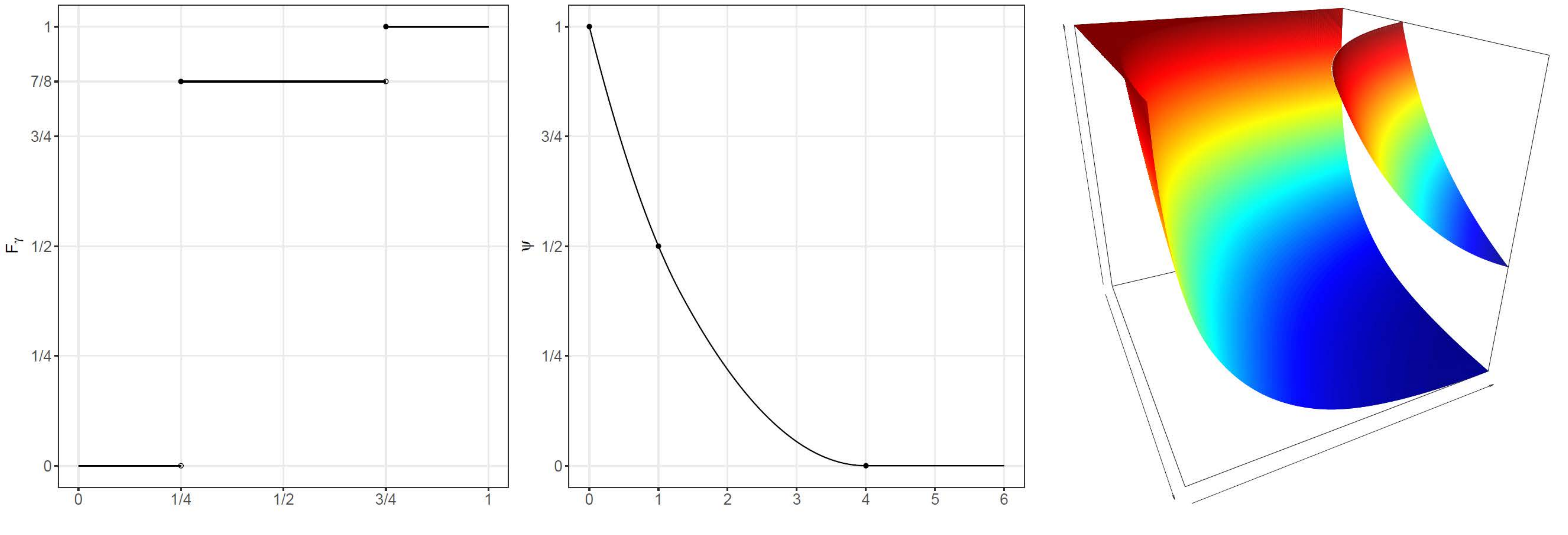}
		\caption{Distribution function of $\gamma$ (left panel), induced generator $\psi$ (middle) and
		the sets $L_0$ and $L_{7/18}$ (right panel) as considered in Example \ref{ex:level:set}.}
		\label{fig:Fgamma:psi}
	\end{figure} 
\end{Ex}

The next result complements Theorem \ref{th:main} and adds a seventh equivalent condition in terms of the corresponding 
Williamson measures. During the process of preparing this manuscript it has been brought to our attention 
that this very result was already established in \cite{Bacigal2017}. Considering that the result is key 
especially for the subsequent regularity results and that the subsequent proof 
is simpler and less technical than the one given in \cite{Bacigal2017} we include it for the sake of completeness.
\begin{theorem}\label{thm:weak.convergence.measures}
Suppose that $C,C_1,C_2,\ldots$ are $d$-dimensional Archimedean copulas with generators 
$\psi, \psi_1, \psi_2,\ldots$ and let $\gamma, \gamma_1, \gamma_2,\ldots$ denote the corresponding Williamson 
measures. Then the following assertions are equivalent:
\begin{itemize}
    \item[(1)] $(C_n)_{n\in\mathbb{N}}$ converges uniformly to $C$.
    \item[(2)] $(\gamma_n)_{n \in \mathbb{N}}$ converges weakly on $[0,\infty)$ to $\gamma$.
\end{itemize}
\end{theorem}
\begin{proof}
According to Theorem \ref{th:main} the first assertion is equivalent to uniform convergence of 
$(\psi_n)_{n \in \mathbb{N}}$ to $\psi$. (i) If $(\gamma_n)_{n \in \mathbb{N}}$ converges weakly to $\gamma$, 
then applying Theorem \ref{thm:will_trans} and using the fact that the function $t \mapsto (1-tz)_+^{d-1}$ is
continuous and bounded 
$$
\psi_n(z) = \int_{[0,\infty)}(1-tz)_+^{d-1} d\gamma_n(t) \overset{n\rightarrow \infty}{\longrightarrow} \int_{[0,\infty)}(1-tz)_+^{d-1} d\gamma(t) =\psi(z)
$$
follows for every fixed but arbitrary $z \in [0,\infty)$. 
(ii) Vice versa, using Lemma \ref{lem:one_dim_measure} and Lemma \ref{lem:conv:derivatives} and considering 
$z \in (0,\infty)$ with $\frac{1}{z} \in \text{Cont}(D^-\psi^{(d-2)})$ yields
\begin{align*}
\lim_{n \rightarrow \infty} \gamma_n([0,z]) &= \lim_{n \rightarrow \infty} \sum_{k=0}^{d-2}\frac{(-1)^k\psi_n^{(k)}(\tfrac{1}{z})}{k!}z^{-k} + \frac{(-1)^{d-1}D^-\psi_n^{(d-2)}(\tfrac{1}{z})}{(d-1)!}z^{-d+1} \\ 
 &= \sum_{k=0}^{d-2}\frac{(-1)^k\psi^{(k)}(\tfrac{1}{z})}{k!}z^{-k} + \frac{(-1)^{d-1}D^-\psi^{(d-2)}(\tfrac{1}{z})}{(d-1)!}z^{-d+1} \\&= \gamma([0,z]).
\end{align*}
This completes the proof since, firstly, $(0,1) \setminus \text{Cont}(D^-\psi^{(d-2)})$ is at most countably infinite 
and, secondly, convergence of distribution functions on a dense set implies weak convergence. 
\end{proof}

We now focus on studying how regularity/singularity properties of the Williamson measure carries over to 
regularity/singularity properties of the corresponding copula $C_\gamma \in \mathcal{C}^d_{ar}$ and first recall some 
basic notation. 
For every $m \in \{1,\ldots,d\}$ we say that a finite measure $\vartheta$ on $\mathcal{B}(\mathbb{I}^m)$ 
is singular (with respect to $\lambda_m$) if there exists some $G \in \mathcal{B}(\mathbb{I}^m)$ fulfilling 
$\vartheta(G) = \vartheta(\mathbb{I}^m)$ and $\lambda_d(G) = 0$. 
A copula $C \in \mathcal{C}^m$ is called singular if the corresponding $m$-stochastic measure $\mu_{C}$ is singular. 

For the bivariate setting singularity of a copula $C$ is equivalent to singularity of $\lambda$-almost all
conditional distributions $K_C(x,\cdot)$ (see Lemma 1 in \cite{Trutschnig2015}). A a fully analogous statement 
can not hold in general for arbitrary $d \geq 3$ - in fact, for example the copula $C$ of a random vector $(X,X,Y)$ with 
$X,Y$ being independent and uniformly distributed on $[0,1]$ is obviously singular but 
$\mu_{C^{1:2}}=\mu_M$-almost every conditional distribution $K_C(x,x,\cdot)$ coincides with $\lambda$ and therefore is 
absolutely continuous. Assuming, however, absolute continuity of 
$C^{1:d-1}$ as it is the case for every $C \in \mathcal{C}^d_{ar}$, an analogue of the bivariate result remains valid:
\begin{Lemma}\label{lem:ker_cop_sing}
Let $C$ be a $d$-dimensional copula such that $C^{1:d-1}$ is absolutely continuous. 
Then $C$ is singular if, and only if there exists some set $\Lambda \in \mathcal{B}(\mathbb{I}^{d-1})$ 
fulfilling $\mu_{C^{1:d-1}}(\Lambda) = 1$ such that $K_C(\mathbf{x},\cdot)$ is singular for every 
$\mathbf{x} \in \Lambda$. 
\end{Lemma}
\begin{proof}
If $C \in \mathcal{C}^d$ is singular then by definition there exists some set 
$G \in \mathcal{B}(\mathbb{I}^{d-1})$ fulfilling $\mu_C(G) = 1$ as well as $\lambda_d(G) = 0$. 
Since every set $G \in \mathcal{B}(\mathbb{I}^{d-1})$ with $\lambda_{d-1}(G)=1$ also fulfills 
$\mu_{C^{1:d-1}}(G)=1$, disintegration implies the existence of a set 
$\Lambda \in \mathcal{B}(\mathbb{I}^{d-1})$ with $\lambda_{d-1}(\Lambda) = 1$ such that for every $\mathbf{x} \in \Lambda$ 
we have that the $\mathbf{x}$-cut $G_\mathbf{x}$ of $G$ fulfills $K_C(\mathbf{x},G_\mathbf{x}) = 1$ and $\lambda(G_\mathbf{x}) = 0$.
In other words, $K_C(\mathbf{x},\cdot)$ is singular and the first implication is proved. 

The reverse implication can be proved as follows. We show the contraposition and assume that 
$\mu_C$ is not singular with respect to $\lambda_d$, i.e., the absolutely continuous part $\mu_C^{abs}$ 
of the Lebesgue decomposition $\mu_C=\mu_C^{abs} + \mu_C^{sing}$ of $\mu_C$ with respect to $\lambda_d$ 
is non-degenerated in the sense that
$\mu_C^{abs}(\mathbb{I}^{d})>0$ holds. Let $G \in \mathcal{B}(\mathbb{I}^{d-1})$ with $\lambda_{d-1}(G) = 0$
be arbitrary but fixed. Then obviously 
$$
(\mu_{C}^{sing})^{1:d-1}(G) = \mu_C^{sing}(G\times \mathbb{I}) \leq \mu_C(G\times \mathbb{I}) = \mu_{C}^{1:d-1}(G) = \int_G c^{1:d-1}(\mathbf{x})\mathrm{d}\lambda_{d-1} = 0,
$$
so there exists a Radon-Nikodym derivative $f\colon\mathbb{I}^{d-1} \rightarrow [0,\infty)$ of 
$(\mu_C^{sing})^{1:d-1}$ with respect to $\lambda_{d-1}$. Letting $k\colon\mathbb{I}^{d} \rightarrow [0,\infty)$ 
denote the Radon-Nikodym derivative of $\mu_C^{abs}$ with respect to $\lambda_{d}$, using disintegration,  
for arbitrary $E \in \mathcal{B}(\mathbb{I}^{d-1})$, $F \in \mathcal{B}(\mathbb{I})$ we get  
\begin{align*}
\int_{E}K_C(\mathbf{x},F)c^{1:d-1} & (\mathbf{x})\mathrm{d}\lambda_{d-1}(\mathbf{x}) = 
\int_{E}K_C(\mathbf{x},F)\mathrm{d}\mu_{C^{1:d-1}}(\mathbf{x})\\& =\mu_C(E \times F) = 
\mu_C^{abs}(E\times F) + \mu_C^{sing}(E\times F) \\&=
\int_E \left[\int_F k(\mathbf{x},y)\mathrm{d}\lambda(y)\right]\mathrm{d}\lambda_{d-1}(\mathbf{x}) + \int_E H^{sing}(\mathbf{x},F)\mathrm{d}(\mu_C^{sing})^{1:d-1}(\mathbf{x}) \\&=
\int_E \left[\int_F k(\mathbf{x},y)\mathrm{d}\lambda(y)\right]\mathrm{d}\lambda_{d-1}(\mathbf{x}) + \int_E H^{sing}(\mathbf{x},F)f(\mathbf{x})\mathrm{d}\lambda_{d-1}(\mathbf{x}) \\&=
\int_E \left[\int_F k(\mathbf{x},y)\mathrm{d}\lambda(y) +  H^{sing}(\mathbf{x},F)f(\mathbf{x})\right]\mathrm{d}\lambda_{d-1}(\mathbf{x})
\end{align*}
where $H^{sing}(\mathbf{x},\cdot)$ denotes the conditional measure (sub- or super Markov kernel) of 
$\mu_C^{sing}$ given $\mathbf{x}$. Since $E \in \mathcal{B}(\mathbb{I}^{d-1})$ it follows that 
$$
K_C(\mathbf{x},F)c^{1:d-1}(\mathbf{x}) = \int_F k(\mathbf{x},y)\mathrm{d}\lambda(y) +  H^{sing}(\mathbf{x},F)f(\mathbf{x})
$$
holds for $\lambda_{d-1}$-almost every $\mathbf{x} \in \mathbb{I}^{d-1}$. Using the fact that
 $\mu_C^{1:d-1}$ is absolutely continuous and that obviously 
 $\mu_{C^{1:d-1}}(\{\mathbf{x}\in\mathbb{I}^{d-1}\colon c^{1:d-1}(\mathbf{x}) = 0\}) = 0$ yields the identity
$$
K_C(\mathbf{x},F) = \int_F \frac{k(\mathbf{x},y)}{c^{1:d-1}(\mathbf{x})}\mathrm{d}\lambda(y) +  H^{sing}(\mathbf{x},F)\frac{f(\mathbf{x})}{c^{1:d-1}(\mathbf{x})}
$$
for $\mu_{C^{1:d-1}}$-almost every $\mathbf{x} \in \mathbb{I}^{d-1}$.
Since $\mu_C^{abs}$ is non-degenerated by assumption, there exists a set 
$\Upsilon \in \mathcal{B}(\mathbb{I}^{d-1})$ with $c^{1:d-1}(\mathbf{x})>0$ for every $\mathbf{x} \in \Upsilon$ and
$\mu_{C^{1:d-1}}(\Upsilon)>0$ such that for every $\mathbf{x} \in \Upsilon$ the absolutely continuous measure 
$F \mapsto \int_F \frac{k(\mathbf{x},y)}{c^{1:d-1}(\mathbf{x})}\mathrm{d}\lambda(y)$ is non-degenerated.
This shows that for such $\mathbf{x}$ the measure $K_C(\mathbf{x},\cdot)$ can not be singular and the proof 
is complete.
\end{proof}
Following \cite{Lange1973} every Markov kernel 
$K_C(\cdot,\cdot) \colon \mathbb{I}^{d-1} \times \mathcal{B}(\mathbb{I}) \rightarrow \mathbb{I}$ can be decomposed into the 
sum of three sub- Markov kernels from $\mathbb{I}$ to $\mathcal{B}(\mathbb{I})$ as
\begin{align}\label{eq:Kernel.decomp}
K_C(\mathbf{x},\cdot) = K_C^{dis}(\mathbf{x},\cdot) + K_C^{sing}(\mathbf{x},\cdot) + K_C^{abs}(\mathbf{x},\cdot),
\end{align}
whereby each measure $K_C^{dis}(\mathbf{x},\cdot)$ is discrete, each $K_C^{sing}(\mathbf{x},\cdot)$ is singular 
and has no point masses and $K_C^{abs}(\mathbf{x},\cdot)$ is absolutely continuous on $\mathcal{B}(\mathbb{I})$.  
Again assuming absolute continuity of $C^{1:d-1}$ and letting $c^{1:d-1}$ denote the corresponding density 
in what follows we will refer to the three measures
$\mu_C^{dis}, \mu_C^{sing}, \mu_C^{abs}$, defined by 
\begin{align}\label{eq:decom.disint}
\mu_C^{dis}(G) &= \int_{\mathbb{I}^{d-1}} K_C^{dis}(\mathbf{x},G_{\mathbf{x}}) c^{1:d-1}(\mathbf{x}) \mathrm{d}\lambda_{d-1}(\mathbf{x}) \nonumber \\
\mu_C^{sing}(G) &= \int_{\mathbb{I}^{d-1}} K_C^{sing}(\mathbf{x},G_{\mathbf{x}}) c^{1:d-1}(\mathbf{x}) \mathrm{d}\lambda_{d-1}(\mathbf{x}) \\
\mu_C^{abs}(G) &= \int_{\mathbb{I}^{d-1}} K_C^{abs}(\mathbf{x},G_{\mathbf{x}}) c^{1:d-1}(\mathbf{x}) \mathrm{d}\lambda_{d-1}(\mathbf{x})   \nonumber
\end{align}
for every $G \in \mathcal{B}(\mathbb{I})$ as the discrete, the singular, and the absolutely continuous component of 
$\mu_C$.  

We now show how singularity/regularity of $\gamma$ carries over to the corresponding Archimedean copula.
\begin{theorem}\label{thm:abs.cont.sing.dis}
Suppose that $C \in \mathcal{C}_{ar}^d$ has generator $\psi$ and Williamson measure $\gamma \in \mathcal{P}_{\mathcal{W}_d}$. Then the following 
assertions hold: 
\begin{itemize}
    \item[(1)] If $\gamma$ is absolutely continuous then $\mu_{C}^{abs}(\mathbb{I}^d)=1$, i.e., $C$ is absolutely continuous. 
    \item[(2)] If $\gamma$ is discrete then $\mu_{C}^{dis}(\mathbb{I}^d)=1$.  
    \item[(3)] If $\gamma$ is singular without point masses then $\mu_{C}^{sing}(\mathbb{I}^d)=1$. 
\end{itemize}
\end{theorem}
\begin{proof}
(i) The first assertion has already been established in \cite{multiArchNeslehova} and can alternatively be proved 
easily as follows: Suppose that $\gamma$ is absolutely continuous with density $f$. Then using Lemma \ref{lem:Left_hand_G}
we have 
$$
(-1)^{d-2}D^{-}\psi^{(d-2)}(z) = -(d-1)!\int_{(0,\frac{1}{z}]}t^{d-1} f(t) \mathrm{d}\lambda(t).
$$  
Considering that the right-hand side is obviously continuous in $z$ it follows that $\psi^{(d-1)}$ exists on 
the full interval $(0,\infty)$. Moreover, the right-hand side is easily seen to be absolutely continuous, hence
Proposition 4.2 in \cite{multiArchNeslehova} yields absolute continuity of $C$.  \\
For the proof of the remaining two assertions first notice that is suffices to consider $\mathbf{x}$ fulfilling
$M(\mathbf{x})<1$ and $\mathbf{x} \not \in L_0^{1:d-1}$ and that in this case we have 
$0< \sum_{i=1}^{d-1} \varphi(x_i)< \varphi(0) \in (0,\infty]$ as well as  
$(-1)^{d-1}D^- \psi^{(d-2)} \left(\sum_{i=1}^{d-1} \varphi(x_i) \right)>0$. Moreover, additionally assuming 
$y \geq f^0(\mathbf{x})$ using Lemma \ref{lem:Left_hand_G}
the Markov kernel $K_C(\cdot,\cdot)$ according to eq. (\ref{eq:mk1}) can be expressed as  
\begin{align}\label{eq:kern.gamma}
K_C(\mathbf{x},[0,y]) &= \frac{\int_{I_y}t^{d-1} \mathrm{d} \gamma(t)}{\int_{I_1}t^{d-1} \mathrm{d} \gamma(t)}
\end{align} 
with $I_y=\left(0, \frac{1}{\sum_{i=1}^{d-1} \varphi(x_i)\, + \varphi(y)} \right]$ for every $y \in \mathbb{I}$. \\
(ii) Suppose now that $\gamma$ is discrete. Then there exist $a_1,a_2,\ldots \in (0,\infty)$ and constants
$\alpha_1,\alpha_2,\ldots \in \mathbb{I}$ with $\sum_{j=1}^\infty \alpha_j=1$ such that 
$\gamma=\sum_{j=1}^\infty \alpha_j \delta_{a_j}$ holds, and eq. (\ref{eq:kern.gamma}) simplifies to 
\begin{align}\label{eq:kern.gamma.temp1}
K_C(\mathbf{x},[0,y]) &= \frac{\sum_{j: a_j \in I_y} a_j^{d-1} \alpha_j}{\sum_{j: a_j \in I_1} a_j^{d-1} \alpha_j}.
\end{align}
Notice that we do not assume all $\alpha_j$ to be greater than zero, so the case of finitely many point masses is 
covered as well. 
Considering, firstly, that $K_C(\mathbf{x},[0,y])=0$ for $y < f^0(\mathbf{x})$ and that 
the condition $a_j \in I_y$ is equivalent to $y \geq \psi(\frac{1}{a_j} - \sum_{i=1}^{d-1} \varphi(x_i))$
the function $y \mapsto K_C(\mathbf{x},[0,y])$ is easily seen to be the distribution function of the 
discrete probability measure on $\mathbb{I}$ having point mass 
$\frac{a_j^{d-1} \alpha_j}{\sum_{l: a_l \in I_1} a_l^{d-1} \alpha_l}$ in 
$\psi(\frac{1}{a_j} - \sum_{i=1}^{d-1} \varphi(x_i))$ for every $j$ with $a_j \in I_1$. 
In other words, $K_C(\mathbf{x},\cdot)$ is a discrete 
probability measure, so $K_C(\mathbf{x},\cdot) = K_C^{dis}(\mathbf{x},\cdot)$ holds 
for all $\mathbf{x}$ fulfilling $M(\mathbf{x})<1$ and $\mathbf{x} \not \in L_0^{1:d-1}$. 
Having this and using eq. (\ref{eq:decom.disint}) $\mu_C^{dis}(\mathbb{I}^d)=1$ follows. \\
(iii) Finally suppose that $\gamma$ is singular without point masses and again consider 
some $\mathbf{x}$ fulfilling $M(\mathbf{x})<1$ and $\mathbf{x} \not \in L_0^{1:d-1}$. 
Then Theorem \ref{thm:level.sets.williamson}
implies that we have $\mu_C(L_t)=0$ for every $t \in \mathbb{I}$ so the 
conditional distribution function $y \mapsto F^C_\mathbf{x}(y)=K_C(\mathbf{x},[0,y])$ is continuous
and it suffices to show that its derivative $(F^C_\mathbf{x})'$ fulfills $(F^C_\mathbf{x})'(y)=0$ 
for $\lambda$-almost every $y > f^0(\mathbf{x})$, which can be done as follows:
As already mentioned before, our choice of $\mathbf{x}$ implies that 
$0< \sum_{i=1}^{d-1} \varphi(x_i)< \varphi(0) \in (0,\infty]$, so in particular 
$$
\frac{1}{\sum_{i=1}^{d-1} \varphi(x_i)} > \frac{1}{\varphi(0)}
$$ 
and therefore $\gamma(I_1)>0$. Defining the measure $m: \mathcal{B}((0,\infty)) \rightarrow [0,\infty]$
by
$$
m(B):= \int_B t^{d-1} \mathrm{d} \gamma(t)
$$
it follows that $m$ is $\sigma$-finite (in fact, finite for every finite interval), singular with respect to 
$\lambda$, has no point masses and fulfills $0<m(I_1)< \infty$. Letting $G_m: I_1 \rightarrow [0,\infty)$ 
denote the measure-generating function induced by $m$ via $G_m(x):=m([0,x])$ singularity of $m$ implies that 
the set $\Lambda$, defined by
$$
\Lambda:=\{z \in I_1: G_m'(z)=0\}
$$
fulfills $\lambda(\Lambda)=\lambda(I_1)$. Defining $\Upsilon$ by
$$
\Upsilon := \left\{y \in (f^0(\mathbf{x}),1]: \, \frac{1}{\sum_{i=1}^{d-1} \varphi(x_i) + \varphi(y)} \in \Lambda\right\} \in \mathcal{B}(\mathbb{I}),
$$
using the fact that $\varphi$ is differentiable and strictly decreasing on $(0,1)$ with derivative 
bounded away from $0$ on any compact interval $[a,b] \subseteq (0,1)$ it follows (see \cite{kannan}) that
$\lambda(\Upsilon)=\lambda((f^0(\mathbf{x}),1])$. For every $y \in \Upsilon$, however, the chain rule yields
$$
(F^C_\mathbf{x})'(y) = \frac{1}{m(I_1)} \,\underbrace{G_m'\left( \frac{1}{\sum_{i=1}^{d-1} \varphi(x_i) + \varphi(y)} \right)}_{=0} \cdot \frac{\partial}{\partial y}\left( \frac{1}{\sum_{i=1}^{d-1} \varphi(x_i) + \varphi(y)} \right) =0.
$$  
Altogether we have shown that for arbitrary $\mathbf{x}$ fulfilling $M(\mathbf{x})<1$ and $\mathbf{x} \not \in L_0^{1:d-1}$
the measure $K_C(\mathbf{x},\cdot)$ is singular without point masses, i.e., 
$K_C(\mathbf{x},\cdot)=K_C^{sing}(\mathbf{x},\cdot)$ holds and considering 
eq. (\ref{eq:decom.disint}) again $\mu_C^{sing}(\mathbb{I}^d)=1$ follows.
\end{proof}

\begin{Rem}
The second assertion of Theorem \ref{thm:abs.cont.sing.dis} can be proved in the following alternative way
(the afore-mentioned version was chosen in order to underline the similarity of the discrete and the singular case): 
Let $\gamma=\sum_{j \in J} \alpha_j \delta_{a_j}$ for some finite our countably infinite index set 
$J \subseteq \mathbb{N}$ where $\alpha_j>0$ for every $j \in J$, and $\sum_{j \in J} \alpha_j=1$ (without loss of generality
we assume $a_i \neq a_j$ for $i \neq j$). Then for every $j \in J$ there exists a unique $t_j \in [0,1)$
with $\frac{1}{\varphi(t_j)}=a_j$ and according to Theorem \ref{thm:level.sets.williamson} we have 
$$
\mu_C \left(\bigcup_{j \in J} L_{t_j}\right) = \sum_{j\in J} \mu_C(L_{t_j}) = \sum_{j\in J} \gamma 
\left(\left\{\frac{1}{\varphi(t_j)} \right\} \right) =1.
$$
The set $L:=\bigcup_{j \in J} L_{t_j}$ is as at most countable union of Borel sets itself 
an element of $\mathcal{B}(\mathbb{I}^d)$. 
Applying disintegration we have 
\begin{align*}
1 = \mu_C(L) = \int_{\mathbb{I}^{d-1}} K_C(\mathbf{x},L_{\mathbf{x}}) \mathrm{d}\mu_{C^{1:d-1}}(\mathbf{x}),
\end{align*}
from which $K_C(\mathbf{x},L_{\mathbf{x}})=1$ for $\mu_{C^{1:d-1}}$-almost every $\mathbf{x}$ follows. 
Considering that the $\mathbf{x}$-cut $L_\mathbf{x}$ of $L$ is at most countably infinite 
we get $K_C^{dis}(\mathbf{x},L_{\mathbf{x}})=1$ for $\mu_{C^{1:d-1}}$-almost every $\mathbf{x}$ from which the 
desired result follows. 
\end{Rem}

\noindent Theorem \ref{thm:abs.cont.sing.dis} has the following consequence, whereby we will let
$\mathcal{C}_{ar,abs}^d$ denote the family of all absolutely continuous $d$-dimensional Archimedean copulas, 
$\mathcal{C}_{ar,dis}^d$ the family of all $C \in \mathcal{C}_{ar}^d$ with $\mu_C^{dis}(\mathbb{I}^d)=1$,
and $\mathcal{C}_{ar,sing}^d$ the family of all $C \in \mathcal{C}_{ar}^d$ with $\mu_C^{sing}(\mathbb{I}^d)=1$. 
\begin{Cor}\label{cor:dense}
$\mathcal{C}_{ar,dis}^d, \, \mathcal{C}_{ar,abs}^d$ and $\mathcal{C}_{ar,sing}^d$ are dense in $(\kc_{ar}^d,d_\infty)$.
\end{Cor}
\begin{proof}
Let $C$ be an arbitrary Archimedean copula and $\gamma$ denote its corresponding Williamson measure on 
$\mathcal{B}([0,\infty))$. The stated results now follow from the fact (see Theorem \ref{thm:approx.scheme}) that 
$\gamma$ is the weak limit of a sequence of discrete, of a sequence of absolutely continuous and of a sequence of 
singular Williamson measures in combination with Theorem \ref{thm:weak.convergence.measures}
and Theorem \ref{thm:abs.cont.sing.dis}. 
\end{proof}

\section{Singular Archimedean copulas with full support}\label{sec7:examples}
The results established in the previous section allow to prove the existence of multivariate Archimedean copulas
which, considering their handy analytic form, exhibit a surprisingly irregular behavior. 
In fact, we will construct singular $d$-dimensional Archimedean copulas with full support $\mathbb{I}^d$
and thereby extend the examples given in \cite{p21} to the multivariate setting. As in the previous section 
the representation in terms of Williamson measures will play a crucial role. 
We first focus on the construction of some $C \in \mathcal{C}_{ar}^d$ fulfilling that $C$ has full support although 
$\mu_{C^{1:d-1}}$-almost every conditional distribution $K_C(\mathbf{x},\cdot)$ is a singular measure 
without point masses and then discuss the discrete analogue.
\begin{theorem}\label{thm:arch.cop.full.supp}
There exists a copula $C \in \mathcal{C}_{ar}^d$ with the following properties:
\begin{itemize}
\item[(1)] $C$ is singular continuous and has full support.
\item[(2)] For $\mu_{C^{1:d-1}}$-almost every $\mathbf{x} \in \mathbb{I}^{d-1}$ the conditional distribution function 
$y \mapsto K_C(\mathbf{x},[0,y]) $ is continuous, strictly increasing and singular. 
\item[(3)] All level sets $L_t$ of $C$ fulfill $\mu_C(L_t) = 0$.
\item[(4)] The Kendall distribution function $F_K^d$ of $C$ is continuous, strictly increasing and singular. 
\end{itemize}
\end{theorem}
\begin{proof}
Suppose that $h$ is a strictly increasing singular homemorphism of $\mathbb{I}$, i.e., a strictly increasing bijective 
transformation mapping $\mathbb{I}$ to itself fulfilling $h'(x)=0$ for $\lambda$-almost every $x \in \mathbb{I}$
(see, e.g., \cite{question,hewitt1965} for several well-known examples).  
Defining $F: [0,\infty) \rightarrow [0,1]$ by
$$
F(x)= \frac{1}{2} \, h\left(\frac{x}{2} \right) \mathbf{1}_{[0,2)}(x) \, + \, 
     \sum_{i=1}^\infty \left(1-\frac{1}{2^i} + \frac{1}{2^{i+1}}\, h\left(\frac{x-2^i}{2^i} \right)\right) 
     \mathbf{1}_{[2^i,2^{i+1})}(x) 
$$ 
obviously yields a strictly increasing continuous function $F$ which, by construction, fulfills $F'=0$ $\lambda$-almost everywhere. Letting $\beta$ denote the corresponding probability measure on $\mathcal{B}([0,\infty))$
it follows that $\beta$ is singular without point masses. Furthermore, the support of $\beta$ contains $0$ 
but in general does not need to be an element of $\mathcal{P}_{\mathcal{W}_d}$, we only know that 
$$
\int_{\mathbb{I}}(1-t)^{d-1} \mathrm{d}\beta(t) \in (0,1).
$$  
Proceeding, however, like in the proof of Lemma \ref{lem:appb1} we can find some constant $a \in (0,\infty)$
such that the push-forward $\gamma=\beta^{T_a}$ with $T_a(x)=ax$ fulfills $\gamma \in \mathcal{P}_{\mathcal{W}_d}$.
Considering that $\gamma$ is obviously singular (without point masses) too and that the support of $\gamma$ 
coincides with $[0,\infty)$ using Lemma \ref{lem:measure.strict.cop} as well as Theorem \ref{thm:abs.cont.sing.dis} it 
follows that the corresponding $d$-dimensional Archimedean copula $C=C_\gamma$ is strict and fulfills 
$\mu_{C_\gamma}^{sing}(\mathbb{I}^d)=1$. Furthermore, according to the proof of Theorem \ref{thm:abs.cont.sing.dis} 
for $\mu_{C^{1:d-1}}$-almost every $\mathbf{x} \in \mathbb{I}^{d-1}$ the conditional distribution function 
$y \mapsto K_C(\mathbf{x},[0,y])$ is continuous and singular. Hence, considering that 
$\gamma$ has full support using eq. (\ref{eq:kern.gamma}) yields that $y \mapsto K_C(\mathbf{x},[0,y])$ is also 
strictly increasing on $\mathbb{I}$. \\
Having that, showing that $C$ has full support is straightforward:
In fact, for every $(\mathbf{x},y) \in (0,1)^{d-1} \times (0,1)$ and every open rectangle 
$U=U_1\times \cdots \times U_d $ with open non-empty intervals $U_1,\ldots,U_d \subseteq (0,1)$ fulfilling
$(\mathbf{x},y) \in U$ and $U \subseteq (0,1)^d$ we can proceed as follows:  
Considering that $C^{1:d-1}$ is absolutely continuous 
strictness of $\psi$ implies that the density $c^{1:d-1}$ of $C^{1:d-1}$ fulfills $c^{1:d-1}>0$ 
$\lambda_{d-1}$-almost everywhere in $(0,1)^{d-1}$ using disintegration and the fact that 
$K_C(\mathbf{x},\cdot)$ has full support and hence fulfills $K_C(\mathbf{x},U_d)>0$ 
for $\mu_{C^{1:d-1}}$-almost every $\mathbf{x} \in \mathbb{I}^{d-1}$
it follows that 
\begin{align*}
\mu_C(U) &= \int_{\times_{j=1}^{d-1} U_j} K_C(\mathbf{x},U_d) \, \mathrm{d}\mu_{C^{1:d-1}}(\mathbf{x}) > 0.
\end{align*}
This shows that $(\mathbf{x},y)$ is contained in the support of $\mu_C$, since supports are closed 
the the support of $\mu_C$ is $\mathbb{I}^d$ and the first two assertions are proved.
Since $\gamma$ has no point masses the third assertion is an immediate consequence of 
Theorem \ref{thm:level.sets.williamson} and it remains to prove the last assertion.  
Again according to Theorem \ref{thm:level.sets.williamson} 
$$
F_K^d(t) = \gamma([0,\tfrac{1}{\varphi(t)}]) = \gamma((0,\tfrac{1}{\varphi(t)}])
$$
holds for every $t \in (0,1]$, implying that $F_K^d$ is continuous and strictly increasing. Finally, using 
a chain rule argument similar to the one at the end of the proof of Theorem \ref{thm:abs.cont.sing.dis}
shows that $(F_K^d)'(x)=0$ holds for $\lambda$-almost every $x \in \mathbb{I}$ and the proof is complete. 
\end{proof}

Starting with the probability measure $\beta:= \sum_{i=1}^\infty 2^{-i} \delta_{q_i}$ with 
$\{q_1,q_2,\ldots\}$ denoting an enumeration of the rationals in $(0,\infty)$ and proceeding 
analogously to the proof of the previous theorem yields the following discrete version of it: 

\begin{theorem}\label{thm:arch.cop.full.supp.dis}
There exists a copula $C \in \mathcal{C}_{ar}^d$ with the following properties:
\begin{itemize}
\item[(1)] $C$ is singular continuous and has full support.
\item[(2)] For $\mu_{C^{1:d-1}}$-almost every $\mathbf{x} \in \mathbb{I}^{d-1}$ the conditional distribution function 
$y \mapsto K_C(\mathbf{x},[0,y])$ is a strictly increasing step function. 
\item[(3)] There exists a dense countable subset $\mathcal{Q}$ of $(0,1)$ such that $\mu_C(L_t)>0$ if, and only if, 
$t \in \mathcal{Q}$.
\item[(4)] The Kendall distribution function $F_K^d$ of $C$ is a strictly increasing step function.
\end{itemize}
\end{theorem}

\appendix

\section{Level set mass and Kendall distribution function: Calculations}

Recall from Section \ref{sec:3:mass:distri} that the $t$-level hypersurfaces $f^t$ are defined on the upper 
$t$-cuts $[C^{1:d-1}]_t$ of the $(d-1)$-marginal. Using the notation $\mathbf{x}_m = (x_1,x_2,\ldots,x_m)$, $m\in\mathbb{N}$, for $\mathbf{x}\in [C^{1:d-1}]_t$ 
we have $x_1\geq t, x_2\geq \psi(\varphi(t)-\varphi(x_1)) =: f^t(x_1), x_3 \geq \psi\left( \varphi(t) - \varphi(x_1) - \varphi(x_2) \right) =: f^t(\mathbf{x}_2)$ and iteratively, 
$$x_{d-1}\geq \psi\big( \varphi(t) - \sum_{i=1}^{d-2}\varphi(x_i) \big) =: f^t(\mathbf{x}_{d-2}).$$

\begin{Prop}\label{prop:level:surfaces}
	Suppose that $C\in\mathcal{C}_\text{ar}^d$ has generator $\psi$ and let $\mu_C$ denote the corresponding $d$-stochastic measure. Then for every $t>0$ we have
	\begin{eqnarray}
		\mu_C(L_t) = \frac{(-\varphi(t))^{d-1}}{(d-1)!} \cdot \big( D^-\psi^{(d-2)}(\varphi(t)) - D^-\psi^{(d-2)}(\varphi(t-)) \big).
	\end{eqnarray}
	If $C$ is strict then $\mu_C(L_0) = 0$ and for non-strict $C$,
	\begin{eqnarray}
		\mu_C(L_0) &= \frac{(-\varphi(0))^{d-1}}{(d-1)!} \cdot D^-\psi^{(d-2)}(\varphi(0)).
	\end{eqnarray}
\end{Prop}
\begin{proof}
	We start with $t>0$. Using disintegration, the definition of $f^t$ and the fact that 
	$C^{1:d-1}$ is absolutely continuous with density $c^{1:d-1}$ we get
	\begin{align*}
				\mu_C(L_t) &= \int_{\mathbb{I}^{d-1}} K_C(\mathbf{s}, (L_t)_\mathbf{s}) \ \mathrm{d}\mu_{C^{1:d-1}}(\mathbf{s}) \\
				&=\int_{[t,1]\times [f^t(s_1),1] \times \ldots \times [f^t(\mathbf{s}_{d-2}),1]} K_C(\mathbf{s}, \{f^t(\mathbf{s})\}) \ \mathrm{d}\mu_{C^{1:d-1}}(\mathbf{s}) \\
				&= \int_{[t,1]} \cdots \int_{[f^t(\mathbf{s}_{d-2}),1]} \prod_{i=1}^{d-1} \varphi'(s_i) \cdot \left[ D^-\psi^{(d-2)}(\varphi(t)) - D^-\psi^{(d-2)}(\varphi(t-)) \right] \ \mathrm{d}\lambda(\mathbf{s}) \\
				&= \left[ D^-\psi^{(d-2)}(\varphi(t)) - D^-\psi^{(d-2)}(\varphi(t-)) \right] \cdot \int_{[t,1]}  \cdots \int_{[f^t(\mathbf{s}_{d-2}),1]} \prod_{i=1}^{d-1} \varphi'(s_i) \ \mathrm{d}\lambda(\mathbf{s}).
	\end{align*}
	\noindent Letting $(II)$ denote the iterated integrals in the previous line we have
	\begin{align*}
		(II) = \int\limits_{[t,1]} \int\limits_{[f^t(s_1),1]} \cdots \prod\limits_{i=1}^{d-2} \varphi'(s_i) \int\limits_{[f^t(\mathbf{s}_{d-2}),1]} \varphi'(s_{d-1}) (-1)^0 \left[ \varphi(t) - \sum\limits_{i=1}^{d-1-0}\varphi(s_i) \right]^0 \ \mathrm{d}\lambda(\mathbf{s})
	\end{align*}
	and the chain rule directly yields
	\begin{align*}
		(II) &= \int_{[t,1]} \cdots \int_{[f^t(\mathbf{s}_{d-4}),1]} \prod\limits_{i=1}^{d-3} \varphi'(s_i) \\
		&\qquad \int_{[f^t(\mathbf{s}_{d-3}),1]} \varphi'(s_{d-2}) \cdot \frac{(-1)^1}{1} \cdot \left[ \varphi(t) - \sum_{i=1}^{d-1-1} \varphi(s_i) \right]^1 \ \mathrm{d}\lambda(s_{d-2}) \mathrm{d}\lambda(\mathbf{s}_{d-3}).
	\end{align*}
	Proceeding analogously for $s_{d-2}$ gives
	\begin{align*}
		(II) &= \int_{[t,1]} \cdots \int_{[f^t(\mathbf{s}_{d-4}),1]} \prod_{i=1}^{d-3} \varphi'(s_i) \cdot \frac{(-1)^2}{1\cdot 2} \cdot \left[ \varphi(t) - \sum_{i=1}^{d-1-2} \varphi(s_i) \right]^2\ \mathrm{d}\lambda(\mathbf{s}_{d-3})
	\end{align*} 
	and after finitely many steps we obtain
	\begin{align*}
		(II) &= \int_{[t, 1]} \varphi'(s_1) \cdot \frac{(-1)^{d-2}}{1\cdot 2 \cdots (d-2)} \left[ \varphi(t) - \varphi(s_1) \right]^{d-2} \ \mathrm{d}\lambda(s_1) = \frac{(-1)^{d-1}}{(d-1)!} \cdot \varphi(t)^{d-1}
	\end{align*}
	as desired. For $t=0$ and strict $C$ we obviously have $\mu_C(L_0) = 0$. For non-strict $C$ we have $K_C(\mathbf{s}, \{ f^0(\mathbf{s}) \}) = K_C(\mathbf{s}, [0, f^0(\mathbf{s})])$ and calculations as those above yield the result.
\end{proof}

\begin{Prop}\label{prop:multivariate:kendall}
	Suppose that $C\in\mathcal{C}_\text{ar}^d$ has generator $\psi$. Then for $t>0$
	\begin{eqnarray}
		F_K^d(t) = D^-\psi^{(d-2)}(\varphi(t))  \frac{(-1)^{d-1}}{(d-1)!} \varphi(t)^{d-1} + \sum_{k=0}^{d-2} \psi^{(k)}(\varphi(t)) \frac{(-1)^k}{k!}\varphi(t)^k.
	\end{eqnarray}
	holds. For $t=0$ and strict $C$ we have $F_K^d(0) = 0$ and for non-strict $C$,
	\begin{eqnarray}
		F_K^d(0) = D^-\psi^{(d-2)}(\varphi(0)) \cdot \frac{(-1)^{d-1}}{(d-1)!}\cdot \varphi(0)^{d-1}.
	\end{eqnarray}
\end{Prop}
\begin{proof}
	Applying disintegration and decomposing $\mathbb{I}^{d-1} = [C^{1:d-1}]_t \cup [C^{1:d-1}]_t^c$ yields 
	\begin{align*}
		F_K^d(t) &= \mu_C([C]_t^c) = \int_{\mathbb{I}^{d-1}} K_C\left(\mathbf{x}, ([C]_t^c)_\mathbf{x}\right) \ \mathrm{d}\mu_{C^{1:d-1}}(\mathbf{x}) \\
		&= \int_{[C^{1:d-1}]_t} K_C\left(\mathbf{x}, ([C]_t^c)_\mathbf{x}\right) \ \mathrm{d}\mu_{C^{1:d-1}}(\mathbf{x})  + \int_{[C^{1:d-1}]_t^c}K_C\left(\mathbf{x}, ([C]_t^c)_\mathbf{x}\right) \ \mathrm{d}\mu_{C^{1:d-1}}(\mathbf{x})
	\end{align*}
	Denoting by $(III)$ and $(IV)$ the first and the second of the above summands, respectively, analogously as in Proposition \ref{prop:level:surfaces} we obtain
	\begin{align*}
		(III) = D^-\psi^{(d-2)}(\varphi(t))\cdot \frac{(-1)^{d-1}}{(d-1)!}\varphi(t)^{d-1}.
	\end{align*}
	Regarding $(IV)$, we have $\mathbf{x}\in [C^{1:d-1}]_t^c$ if, and only if, $([C]_t^c)_\mathbf{x} = \mathbb{I}$ and hence
	\begin{align*}
		(IV) &= \int_{[C^{1:d-1}]_t^c} 1 \ \mathrm{d}\mu_{C^{1:d-1}}(\mathbf{x}) = \mu_{C^{1:d-1}}([C^{1:d-1}]_t^c) = F_K^{d-1}(t).
	\end{align*}
	Proceeding iteratively finally yields
	\begin{align*}
		F_K^d(t) &= D^-\psi^{(d-2)}(\varphi(t)) \cdot \frac{(-1)^{d-1}}{(d-1)!}\cdot \varphi(t)^{d-1}  + \sum_{k=1}^{d-2} \psi^{(k)}(\varphi(t)) \frac{(-1)^k}{k!}\varphi(t)^k + t.
	\end{align*}
	For $t=0$ we have $
	F_K^d(t)= \int_{\{ \mathbf{s}\in\mathbb{I}^{d-1}: \ \sum_{i=1}^{d-1}\varphi(s_i)\leq \varphi(0) \}} K_C(\mathbf{x}, ([C]_0^c)_\mathbf{x}) \ \mathrm{d}\mu_{C^{1:d-1}}(\mathbf{x})
	$
	and the result follows in the same manner.
\end{proof}
\begin{Lemma}\label{convex.cont.conv}
	Suppose that $f,f_1,f_2,\ldots$ are convex functions such that $(f_n)_{n\in\mathbb{N}}$ converges 
	to $f$ pointwise on $(0,\infty)$. Then $\lim_{n\rightarrow\infty}D^{-}f_n(x) = D^{-}f(x)$ holds 
	for every $x\in \text{Cont}(D^{-}f)$. Moreover, the sequence $(D^{-}f_n)_{n\in\mathbb{N}}$ converges 
	continuously to $D^{-}f$ on $\text{Cont}(D^{-}f)$, i.e., for every sequence $(z_n)_{n \in \mathbb{N}}$ with 
	limit $z \in \text{Cont}(D^{-}f)$ we have 
	$$
	\lim_{n \rightarrow \infty} D^{-}f_n(z_n)=D^{-}f(z).
	$$
\end{Lemma}
\begin{proof}
	Convexity implies that for every $h>0$ we have
	\begin{align*}
		\limsup_{n \rightarrow \infty}D^{+}f_n(x) \leq \limsup_{n \rightarrow \infty} 
		\frac{f_n(x+h)- f_n(x)}{h} = 
		\frac{f(x+h)-f(x)}{h},
	\end{align*}
	which, considering $h\downarrow 0$, yields $\limsup_{n \rightarrow \infty}D^+f_n(x) \leq D^+f(x)$. 
	The inequality $\liminf_{n\rightarrow \infty} D^-f_n(x) \geq D^-f(x)$ follows in the same manner, so
	altogether we get
	\begin{align*}
		D^-f(x) \leq \liminf_{n\rightarrow \infty} D^-f_n(x) \leq \limsup_{n \rightarrow \infty} 
		D^+f_n(x) \leq D^+f(x)
	\end{align*}
	from which the fist assertion follows since for $x\in \text{Cont}(D^-f)$ we have 
	$D^+f(x) = D^-f(x)$.
	
	For the second part let $(z_n)_{n\in\mathbb{N}}$ be a sequence in $(0,\infty)$ converging to some point
	$z\in \text{Cont}(D^-f) \subseteq(0,\infty)$. 
	We can find two strictly decreasing sequences 
	$(a_k)_{k\in\mathbb{N}}, (b_k)_{k\in\mathbb{N}}$ in $(0,\infty)$ 
	converging to $0$ with $z-a_k, z+b_k \in \text{Cont}(D^-f)$ for every 
	$k\in\mathbb{N}$. 
	Fix $k \in \mathbb{N}$. Then there exists some index $n_0 \in \mathbb{N}$ such that for all $n\geq n_0$ 
	we have that $z-a_k < z_n < z+b_k$ holds. Monotonicity of $D^-f_n$ implies 
	\begin{align*}
		D^-f_n(z-a_k) \leq D^-f_n(z_n) \leq D^-f_n(z+b_k)
	\end{align*} 
	for every such $n\geq n_0$. Having this we get 
	\begin{align*}
		\lim_{n \rightarrow \infty} D^-f_n(z-a_k) &\leq \liminf_{n \rightarrow \infty} D^-f_n(z_n)  \\ &\leq  \limsup_{n \rightarrow \infty} D^-f_n(z_n) \leq \lim_{n \rightarrow \infty}  D^-f_n(z+b_k),
	\end{align*}
	from which the result follows since $z \in \text{Cont}(D^-f)$ and therefore
	$$
		\lim_{k \rightarrow \infty} \lim_{n \rightarrow \infty}  D^-f_n(z-a_k) = D^-f(z)=
		\lim_{k \rightarrow \infty} \lim_{n \rightarrow \infty}  D^-f_n(z+b_k)
	$$
	holds.
\end{proof}
\section{Approximations by discrete, absolutely continuous and singular Williamson measures}
We now tackle Theorem \ref{thm:approx.scheme} already used in the proof of Corollary \ref{cor:dense} and start with the following simple lemma simplifying the approximation procedure. Thereby, for every $a \in (0,\infty)$
we will let $T_a:(0,\infty) \rightarrow (0,\infty)$ denote the linear transformation $T_a(x)=ax$.
\begin{Lemma}\label{lem:appb1}
Suppose that $\gamma \in \mathcal{P}_{\mathcal{W}_d}$, that $(\beta_n)_ {n \in \mathbb{N}}$ is a sequence 
of probability measures on $\mathcal{B}([0,\infty)$ satisfying $\beta_n(\{0\})=0$ for every $n \in \mathbb{N}$
but not necessarily fulfilling eq. (\ref{eq:cond_meas}), and that $(\beta_n)_ {n \in \mathbb{N}}$ converges
weakly to $\gamma$. Then there exists a sequence $(a_n)_{n \in \mathbb{N}}$
in $(0,\infty)$ converging to $\frac{1}{2}$ such that the following properties hold:
 \begin{itemize}
 \item Each probability measure $\gamma_n:=\beta_n^{T_{a_n}}, n \in \mathbb{N}$,
  fulfills $\gamma_n \in \mathcal{P}_{\mathcal{W}_d}$.
  \item  $(\gamma_n)_ {n \in \mathbb{N}}$ converges weakly on $[0,\infty)$ to $\gamma$.
 \end{itemize}
\end{Lemma}
\begin{proof}
Let $\psi_\gamma$ denote the normalized Archimedean generator corresponding to $\gamma$ and 
$\psi_{\beta_n}$ the (not necessarily normalized) generator induced via 
$\psi_{\beta_n}=\mathcal{W}_d(\beta_n)$. Then proceeding as in the first part of the proof of 
Theorem \ref{thm:weak.convergence.measures} it follows that $(\psi_{\beta_n})_{n \in \mathbb{N}}$
converges uniformly to $\psi_\gamma$. Letting $a_n$ denote the unique element in $(0,\infty)$ fulfilling
$\psi_{\beta_n}(a_n)=\frac{1}{2}$ using monotonicity of generators and the fact that $\psi_\gamma$ is normalized
it is straightforward to verify that $(a_n)_{n \in \mathbb{N}}$ converges to $1$. \\
The probability measure $\gamma_n:=\beta_n^{T_{a_n}}$ obviously fulfills $\gamma_n(\{0\})=0$, moreover
using change of coordinates yields 
\begin{align*}
\frac{1}{2} &= \psi_{\beta_n}(a_n) = \int_{(0,\infty)} (1-ta_n)_+^{d-1} \,\mathrm{d} \beta_n(t) =  
\int_{(0,\infty)} (1-T_{a_n}(t))_+^{d-1} \,\mathrm{d} \beta_n(t) \\
&= \int_{(0,\infty)} (1-s)_+^{d-1} \,\mathrm{d} \gamma_n(t) = \int_{(0,1)} (1-s)^{d-1} \,\mathrm{d} \gamma_n(t),
\end{align*}  
so $\gamma_n \in \mathcal{P}_{\mathcal{W}_d}$ and it remains to show weak convergence. 
Applying Lemma \ref{lem:one_dim_measure} and Lemma \ref{lem:conv:derivatives} (continuous convergence) yields 
\begin{align*}
\lim_{n\rightarrow \infty}\gamma_n([0,z]) &= \lim_{n\rightarrow \infty} \beta_n\left(\left[0,\frac{z}{a_n}\right]\right)
\\&= \lim_{n\rightarrow \infty}\sum_{k = 0}^{d-2}\frac{(-1)^k\psi_{\beta_n}^{(k)}(\tfrac{a_n}{z})}{k!}\left(\frac{a_n}{z}\right)^k + \frac{(-1)^{d-1}D^-\psi_{\beta_n}^{(d-2)}(\tfrac{a_n}{z})}{(d-1)!}\left(\frac{a_n}{z}\right)^{d-1} \\&=
\gamma([0,z])
\end{align*}
for every point $z \in (0,\infty)$ with $\gamma(\{z\})=0$ which completes the proof.
\end{proof}
  
\begin{theorem}\label{thm:approx.scheme}
Suppose that $\gamma \in \mathcal{P}_{\mathcal{W}_d}$. Then there exists a sequence $(\gamma_n^{1})_{n \in \mathbb{N}}$
of discrete measures in $\mathcal{P}_{\mathcal{W}_d}$, a sequence $(\gamma_n^2)_{n \in \mathbb{N}}$
of singular measures without point masses in $\mathcal{P}_{\mathcal{W}_d}$, and 
a sequence $(\gamma_n^3)_{n \in \mathbb{N}}$ of absolutely continuous measures $\mathcal{P}_{\mathcal{W}_d}$ 
that all converge weakly to $\gamma$ on $[0,\infty)$.
\end{theorem}
\begin{proof}
Let $F$ denote the distribution function corresponding to $\gamma$, i.e., $F(z)=\gamma((0,z])=\gamma([0,z])$
for every $z \in [0,\infty)$ and let $\mathcal{Q}:=\{q_0,q_1,q_2,\ldots\}$ denote a countably infinite 
subset of Cont$(F)$ which is dense in $[0,\infty)$. Without loss of generality we assume that $q_0=0$. 
Furthermore let the function $f: \mathbb{I} \rightarrow [0,1]$ be right-continuous, non-decreasing with
$f(0)=0, f(1)=1$ and $g: [0,\infty) \rightarrow \mathbb{I}$ be right-continuous, non-decreasing with
$g(0)=0, g(\infty)=1$. For every non-degenerated compact interval $[a,b] \subseteq [0,\infty)$ and compact 
interval $[c,d] \subseteq [0,1]$ define the rescaled version $f_{[a,b]}^{[c,d]}: [a,b] \rightarrow [c,d]$ of
$f$ to $[a,b], [c,d]$ by 
$$
f_{[a,b]}^{[c,d]}(x) = c + (d-c) f \left(\frac{x-a}{b-a} \right),
$$     
and for every interval $[a,\infty) \subseteq [0,\infty)$ and $[c,d] \subseteq \mathbb{I}$ define 
$g_{[a,\infty)}^{[c,d]}: [a,\infty) \rightarrow [c,d]$ of $g$ to $[a,\infty), [c,d]$ by  
$$
g_{[a,\infty)}^{[c,d]}(x) = c + (d-c) g(x-a).
$$  
Using this notation define the distribution function $F_1:[0,\infty) \rightarrow [0,1]$ by 
$$
F_1(x)=f_{[0,q_1]}^{[0,F(q_1)]}(x) \cdot \mathbf{1}_{[0,q_1]}(x) \, + \, g_{[q_1,\infty)}^{[F(q_1),1]}(x) 
\cdot \mathbf{1}_{(q_1,\infty)}(x)
$$
and notice that $F_1$ fulfills $F_1(q_i)=F(q_i)$ for $i \in \{0,1\}$. In the second step  
define the distribution function $F_2:[0,\infty) \rightarrow [0,1]$ by 
$$
F_2(x)=f_{[0,q^2_{(1)}]}^{[0,F(q^2_{(1)})]}(x) \cdot \mathbf{1}_{[0,q^2_{(1)}]}(x) \, + \, 
f_{[q^2_{(1)},q^2_{(2)}]}^{[F(q^2_{(1)}), F(q^2_{(2)})]}(x) \cdot \mathbf{1}_{(q^2_{(1)},q^2_{(2)}]}(x)
 \, + \, g_{[q^2_{(2)},\infty)}^{[F(q^2_{(2)}),1]}(x) \cdot \mathbf{1}_{(q^2_{(2)},\infty)}(x),
$$
whereby $0< q^2_{(1)} < q^2_{(2)}$ denotes the order statistics of $0,q_1,q_2$ (the exponent denotes the `sample size').
Obviously $F_2$ fulfills $F_2(q_i)=F(q_i)$ for $i \in \{0,1,2\}$. Proceeding analogously yields a
sequence $(F_n)_{n \in \mathbb{N}}$ of distribution functions on $[0,\infty)$ fulfilling that for every 
$i \in \mathbb{N}$ we have $F_n(q_i)=F(q_i)$ for every $n \geq i$. In other words, the sequence 
$(F_n)_{n \in \mathbb{N}}$ converges on a dense set to the distribution function $F$, implying that 
$(F_n)_{n \in \mathbb{N}}$ converges to $F$ weakly. \\
Notice that the just discussed construction works for arbitrary $f,g$ fulfilling the afore-mentioned requirements. 
If we chose both $f$ and $g$ absolutely continuous then obviously
each $F_n$ is absolutely continuous, if we chose both $f$ and $g$ as step functions then $F_n$ is a step function, and
if we choose $f,g$ to be continuous with $f'=0$ and $g'=0$ almost everywhere (one could use, for instance, the 
Cantor function or work with any other strictly increasing singular continuous distribution function, 
see \cite{hewitt1965}) then each $F_n$ is continuous and obviously fulfills $F_n'=0$ almost everywhere.
The desired result now follows by considering the probability measures $\beta_n$ corresponding
to $F_n$ and applying Lemma \ref{lem:appb1}.  
\end{proof}
\section*{Acknowledgements}\vspace*{-0.2cm}
	The first author gratefully acknowledges the financial support from Porsche Holding Austria and 
	Land Salzburg within the WISS 2025 project \textquoteleft KFZ' (P1900123).
	The second and the third author gratefully acknowledge the support of the WISS 2025 project \textquoteleft 
	IDA-lab Salzburg' (20204-WISS/225/197-2019 and 20102-F1901166-KZP).


\end{document}